\documentclass[11pt]{amsart}

\usepackage{amsmath,amsthm,amsfonts,amssymb,amscd,flafter,pinlabel, booktabs, graphics}
\usepackage[mathscr]{eucal}
\usepackage[all]{xy}
\usepackage[usenames,dvipsnames]{xcolor}
\usepackage{subfigure}
\usepackage{graphicx}
\usepackage{pgf}
\usepackage{tikz}
\usetikzlibrary{arrows,automata}
\usepackage[latin1]{inputenc}
\usepackage{pifont}

\definecolor{ao}{rgb}{0.13, 0.55, 0.13}

\usepackage[colorlinks=true, urlcolor=NavyBlue, linkcolor=NavyBlue, citecolor=NavyBlue, pdftitle={Contact Invariants in Heegaard Floer Theory}, pdfauthor={John A. Baldwin, David Shea Vela-Vick}, pdfsubject={Contact Invariants}, pdfkeywords={Open book, transverse braid, Heegaard Floer homology, 57M27, 57R58, 57R17}]{hyperref}

\newtheorem{theorem}{Theorem}[section]
\newtheorem{corollary}[theorem]{Corollary}

\newtheorem{proposition}[theorem]{Proposition}

\newtheorem{lemma}[theorem]{Lemma}

\theoremstyle{definition}

\theoremstyle{definition}
\newtheorem{remark}[theorem]{Remark}

\newcommand{\F}{\mathbb{F}}
\newcommand{\R}{\mathbb{R}}

\newcommand{\xistd}{\xi_{\mathrm{std}}}
\newcommand{\alphastd}{\alpha_{\mathrm{std}}}
\newcommand{\llink}{\Lambda}
\newcommand{\lcob}{L}
\newcommand{\llinkb}{\llink_{-}}
\newcommand{\llinkt}{\llink_{+}}
\newcommand{\llinki}{\llink_{i}}
\newcommand{\tb}{\mathit{tb}}
\newcommand{\rot}{\mathit{r}}

\newcommand{\cmark}{\color{ao}\ding{51}}
\newcommand{\xmark}{\color{red}\ding{55}}

\newcommand{\GHt}{\widetilde{\mathrm{GH}}}
\newcommand{\GHm}{\mathrm{GH}^-}
\newcommand{\GCh}{\widehat{\mathrm{GC}}}
\newcommand{\GCm}{\mathrm{GC}^-}
\newcommand{\GCt}{\widetilde{\mathrm{GC}}}
\newcommand{\GHh}{\widehat{\mathrm{GH}}}
\newcommand{\HFKh}{\widehat{\mathrm{HFK}}}

\newcommand{\HFKm}{\mathrm{HFK}^-}

\newcommand{\AB}{\mathrm{AB}}
\newcommand{\AN}{\mathrm{AN}}
\newcommand{\NB}{\mathrm{NB}}
\newcommand{\NN}{\mathrm{NN}}
\newcommand{\ABt}{\widetilde{\mathrm{AB}}}
\newcommand{\ANt}{\widetilde{\mathrm{AN}}}
\newcommand{\NBt}{\widetilde{\mathrm{NB}}}
\newcommand{\NNt}{\widetilde{\mathrm{NN}}}

\newcommand{\gd}{\mathbb{G}}
\newcommand{\gdb}{\mathbb{G}_-}
\newcommand{\gdt}{\mathbb{G}_+}
\newcommand{\gdo}{\mathbb{O}}
\newcommand{\gdx}{\mathbb{X}}
\newcommand{\Phit}{\Phi}

\newcommand{\lossh}{\widehat{\mathfrak{L}}}

\newcommand{\lgridhp}{\widehat\lambda^+}
\newcommand{\lgridhm}{\widehat\lambda^-}
\newcommand{\lgridhpm}{\widehat\lambda^\pm}

\newcommand{\lgridtp}{\widetilde\lambda^+}
\newcommand{\lgridtm}{\widetilde\lambda^-}
\newcommand{\lgridtpm}{\widetilde\lambda^\pm}

\renewcommand{\SS}{\mathbf{S}}
\newcommand{\Rect}{\mathrm{Rect}}
\newcommand{\Recto}{\mathrm{Rect^o}}
\newcommand{\Pent}{\mathrm{Pent}}
\newcommand{\Pento}{\mathrm{Pent^o}}
\newcommand{\Tri}{\mathrm{Tri}}
\newcommand{\Trio}{\mathrm{Tri^o}}
\newcommand{\bx}{\mathbf{x}}
\newcommand{\by}{\mathbf{y}}

\newcommand{\xp}{\mathbf{x}^+}
\newcommand{\xm}{\mathbf{x}^-}
\newcommand{\xpm}{\mathbf{x}^\pm}

\begin{document}

\title[Lagrangian cobordisms and  Legendrian invariants]{Lagrangian cobordisms and  Legendrian invariants in knot Floer homology}

\author[John A. Baldwin]{John A. Baldwin}
\address{Department of Mathematics \\ Boston College}
\email{\href{mailto:john.baldwin@bc.edu}{john.baldwin@bc.edu}}
\urladdr{\href{https://www2.bc.edu/john-baldwin}{https://www2.bc.edu/john-baldwin}}

\author[Tye Lidman]{Tye Lidman}
\address{Department of Mathematics \\ North Carolina State University }
\email{\href{mailto:tlid@math.ncsu.edu}{tlid@math.ncsu.edu}}
\urladdr{\url{http://www4.ncsu.edu/~tlidman}}

\author[C.-M. Michael Wong]{C.-M. Michael Wong}
\address{Department of Mathematics \\ Louisiana State University}
\email{\href{mailto:cmmwong@lsu.edu}{cmmwong@lsu.edu}}
\urladdr{\url{http://www.math.lsu.edu/~cmmwong}}

\thanks{JAB was partially supported by NSF CAREER Grant DMS-1454865.\\
\indent TL was partially supported by NSF DMS-1709702 and a Sloan Fellowship.}



\begin{abstract}
We prove that  the  LOSS and GRID invariants of Legendrian links in knot Floer homology behave in certain functorial ways  with respect to  decomposable Lagrangian cobordisms in the symplectization of the standard  contact structure on $\R^3$. Our results give new, computable, and effective  obstructions to the existence of such cobordisms.
\end{abstract}

\maketitle

\section{Introduction}

Let $\xistd$ be the standard  contact structure on $\R^3$, given by the kernel of the $1$-form \[\alphastd=dz-ydx. \] 
A     difficult problem in contact and symplectic geometry, which has attracted a great deal of attention in recent years,  is to decide, given two Legendrian links \[\llinkb,\llinkt \subset (\R^3,\xistd),\] whether there exists an exact Lagrangian cobordism  from $\llinkb$ to $\llinkt$ in the symplectization \[(\R_t\times\R^3,d(e^t\alphastd)).\] In the smooth category,  any two  links are cobordant in $\R\times\R^3$, and the challenge is to determine the minimum genus among  such cobordisms. The opposite is true in the Lagrangian setting, where the existence  of an exact Lagrangian cobordism  is constrained but its genus is completely determined by the  classical Thurston--Bennequin and rotation numbers of the Legendrian   links at the ends. Indeed, Chantraine showed in \cite{Cha10:LagConc}  that if  $L$ is an exact Lagrangian cobordism from $\llinkb$ to $\llinkt$,  then   
\begin{equation}
\label{eqn:tbr}\tb(\llinkt)-\tb(\llinkb) = -\chi(L) \textrm{ and } \rot(\llinkt) = \rot(\llinkb).\end{equation} 
An important  goal, therefore, is to develop obstructions to the existence of exact Lagrangian cobordisms that are \emph{effective}, meaning that they can obstruct such cobordisms where smooth topology and the classical invariants do not. 

In this article, we   restrict our attention to \emph{decomposable} Lagrangian cobordisms,  which are those that can be obtained as compositions of  elementary cobordisms associated to Legendrian isotopies, pinches, and births, as shown  in Figure \ref{fig:decomposable}. Decomposable  cobordisms are exact, and constitute most known examples of exact Lagrangian cobordisms. It is  open whether all exact Lagrangian cobordisms are decomposable.

Our main result, described in Sections \ref{ssec:obstructions}-\ref{ssec:examples}, is that knot Floer homology provides effective  obstructions to  decomposable Lagrangian cobordisms. Symplectic Field Theory also furnishes  various   obstructions to exact Lagrangian cobordisms; see \cite{EHK,CorNgSiv16:LagConcObstructions, Pan,CDGGG,ST}. One advantage of our knot Floer   obstructions is that they are generally much  easier to compute  than those coming from  SFT. Moreover, we  show  that knot Floer homology  obstructs decomposable cobordisms in cases where  the SFT invariants do not (and vice versa).

As   discussed in Section \ref{ssec:antecedents}, there is  an existing   body of  work \cite{BalSiv14:KHMLeg, BalSiv18:EqInv,GolJuh18:LOSSConc} showing that knot Floer homology effectively obstructs  Lagrangian cobordisms of genus-zero in various settings. Ours is the first result that shows that  knot Floer homology can effectively obstruct Lagrangian cobordisms of \emph{positive} genus.

\subsection{Obstructions}
\label{ssec:obstructions}

In \cite{OzsSzaThu08:GRID}, Ozsv{\'a}th, Szab{\'o}, and Thurston used the combinatorial  grid diagram formulation of knot Floer homology \cite{ManOzsSar09:GH} to define  invariants of Legendrian links in $ (\R^3,\xistd)$. These so-called GRID invariants assign to  such a Legendrian  link  $\llink$ two  elements in  the \emph{hat} flavor of the knot Floer homology of $\llink\subset -S^3$,\footnote{The GRID invariant is defined for knots in $S^3$, but a Legendrian knot in $(\R^3,\xistd)$ can be viewed naturally as a Legendrian in the standard contact structure on $S^3$.  We follow the conventions of \cite{OzsSzaThu08:GRID} and view these invariants as living in $\HFKh(S^3,m(\Lambda))$, which we identify with $\HFKh(-S^3, \Lambda)$.}
\[\lgridhp(\llink),\lgridhm(\llink)\in\HFKh(-S^3,\llink),\footnote{There are also versions of the GRID invariants in the more general \emph{minus} flavor.}
\] which depend only on the Legendrian isotopy class of $\llink$. These elements are  \emph{effective} invariants in  that they can distinguish Legendrian links that are not isotopic but have the same classical invariants (see \cite{NgOzsThu08:GRIDEffective}, for example), and are combinatorially computable.

\begin{remark}
The  Maslov (or Alexander) gradings of the  classes $\lgridhpm(\llink)$   recover the Thurston--Bennequin and rotation numbers  of $\llink$ (see Section \ref{ssec:hfkgrid}).\end{remark}


We prove that the GRID invariants are well-behaved under decomposable Lagrangian cobordisms. As explained in Section \ref{ssec:examples}, this provides effective obstructions to such cobordisms.

\begin{theorem}
\label{thm:main} Suppose $\llinkb,\llinkt$ are Legendrian links in $(\R^3,\xistd)$ such that either
\begin{itemize}
\item $\lgridhp(\llinkt) = 0$ and $\lgridhp(\llinkb) \neq 0$, or
\item $\lgridhm(\llinkt) = 0$ and $\lgridhm(\llinkb) \neq 0$.
\end{itemize} Then there is no decomposable Lagrangian cobordism from $\llinkb$ to $\llinkt$.
\end{theorem}

Theorem \ref{thm:main} has the following corollaries.

\begin{corollary}
\label{cor:lagfilling}
Suppose $\llink$ is a Legendrian link in $(\R^3,\xistd)$ such that either
\[\lgridhp(\llink) = 0 \textrm{ or }
\lgridhm(\llink) = 0.\]
 Then there is no decomposable Lagrangian filling of $\llink$.
\end{corollary}

As explained in Section \ref{ssec:proofs}, this follows from the fact that the GRID invariants are nonzero for the $tb=-1$ Legendrian unknot.

\begin{corollary}
\label{cor:stab}
Suppose $\llinkb,\llinkt$ are Legendrian links in $(\R^3,\xistd)$ such that either
\begin{itemize}
\item $\lgridhp(\llinkb) \neq 0$ and $\llinkt$ is the positive stabilization of a Legendrian link, or
\item $\lgridhm(\llinkb) \neq 0$ and $\llinkt$ is the negative stabilization of a Legendrian link.
\end{itemize} Then there is no decomposable Lagrangian cobordism from $\llinkb$ to $\llinkt$.
\end{corollary}

This follows immediately from the fact (see Proposition \ref{prop:stab}) that the elements $\lgridhp$ and $\lgridhm$ vanish for positively and negatively stabilized Legendrian links, respectively.


In \cite{LisOzsSti09:LOSS}, Lisca, Ozsv{\'a}th, Stipsicz, and Szab{\'o} used open book decompositions to define a knot Floer  invariant of Legendrian \emph{knots} in any closed contact 3-manifold. For  a Legendrian knot $\llink\subset (\R^3,\xistd)$, their so-called LOSS invariant also  takes the form of an element  \[\lossh(\llink)\in\HFKh(-S^3,\llink).\] Although the LOSS invariant is not algorithmically computable, Baldwin, Vela-Vick, and V{\'e}rtesi proved in  \cite{BalVelVer13:BRAID} that it agrees with the GRID invariants, for  Legendrian knots in $(\R^3,\xistd)$. More precisely, given such a knot $\llink$, there are isomorphisms
\[
\phi_\pm:\HFKh(-S^3,\llink)\to\HFKh(-S^3,\pm\llink)
\]
such that  \[\phi_\pm(\lgridhpm(\llink))=\lossh(\pm\llink).\] This   gives   corresponding versions of Theorem \ref{thm:main} and its corollaries   for the LOSS invariant.

\subsection{Proof}
\label{ssec:proof} Theorem \ref{thm:main} follows from  a similar result for the \emph{tilde} version of the GRID invariants. We explain this below after  providing a bit of additional  background on the construction of the GRID invariants (see Section \ref{ssec:hfkgrid} for details).

A Legendrian link $\llink \subset (\R^3,\xistd)$ can be represented by a grid diagram $\gd$. This grid diagram determines a  combinatorially computable, bigraded chain complex   whose \emph{grid homology} agrees with the knot Floer homology of $\llink\subset -S^3$, \[\GHh(\gd) \cong \HFKh(-S^3,\llink).
\] There are two canonical cycles in this  grid chain complex, representing elements \[\lgridhp(\gd),\lgridhm(\gd)\in\GHh(\gd).\] The  \emph{hat} version of the GRID invariants discussed previously are defined by  \[\lgridhpm(\llink) :=  \lgridhpm(\gd).\] A specialization of this  chain complex gives rise to the \emph{tilde}  version of grid homology, which agrees with the \emph{tilde} flavor of knot Floer homology, and is related to the \emph{hat} flavor by \[\GHt(\gd) \cong \HFKh(-S^3,\llink)\otimes V^{\otimes |\gd|-|\llink|},\]  where \[V= \F_{0,0} \oplus \F_{-1,-1}\] is the two-dimensional vector space supported in the Maslov--Alexander bigradings indicated by the subscripts; $|\gd|$ is  the grid number of $\gd$; and $|\llink|$ is the number of components of $\llink$.  There are two canonical elements in this version of grid homology as well, \[\lgridtp(\gd),\lgridtm(\gd)\in\GHt(\gd),\] which we refer to as the   \emph{tilde} version of the GRID invariants. Moreover, there is an  injection \[\GHh(\gd)\hookrightarrow\GHt(\gd)\] that sends $\lgridhpm(\gd)$ to $\lgridtpm(\gd)$  \cite{NgOzsThu08:GRIDEffective}. In particular, \[\lgridhpm(\llink)=\lgridhpm(\gd)= 0 \textrm{ iff }\lgridtpm(\gd)= 0.\]  Theorem \ref{thm:main} therefore follows immediately  from our main technical result below,  which states that the \emph{tilde} versions of the GRID invariants satisfy a weak functoriality under decomposable Lagrangian cobordisms.

\begin{theorem}
\label{thm:maintilde} Suppose $\llinkb,\llinkt$ are Legendrian links in $(\R^3,\xistd)$ with grid representatives $\gdb,\gdt$, respectively. Suppose there exists a decomposable Lagrangian cobordism $L$ from $\llinkb$ to $\llinkt$. Then there is a  homomorphism
\[\Phit_L:\GHt(\gdt)\to\GHt(\gdb)\footnote{The notation $\Phit_L$ is a slight abuse of notation as we do not prove that this map  depends only on $L$.} \] such that  \[\Phit_L(\lgridtpm(\gdt))=\lgridtpm(\gdb).\] This map has Maslov--Alexander bidegree  \[(\chi(L),\frac{1}{2}(\chi(L)+|\llinkb|-|\llinkt|)),\] where $|\llink_\pm|$ is the number of components of $\llink_\pm$.\end{theorem}

Recall that a decomposable  cobordism $L$ as in the theorem can be described as a composition of elementary cobordisms associated with Legendrian isotopies, pinches, and births. To prove Theorem \ref{thm:maintilde}, we  define combinatorially computable maps on the \emph{tilde} version of grid homology for   each of these elementary cobordisms (the maps corresponding to Legendrian isotopies were  defined  in \cite{OzsSzaThu08:GRID}), and  show that  these  elementary maps preserve the \emph{tilde} GRID invariant. We then define $\Phi_L$ to be the appropriate composition of these elementary maps.

In \cite{Juh16:SFHFunctoriality,Zem19:HFKFunctoriality}, Juh{\'a}sz and Zemke independently proved that decorated link cobordisms between pointed links induce  well-defined maps on knot Floer homology. (They defined these maps differently, but showed in \cite{JuhZem18:HFKFunctorialityEquiv} that their  definitions agree for the \emph{tilde}  flavor of $\mathrm{HFK}$.) A grid diagram  naturally  specifies a pointed  link, and the sequence of grid  moves corresponding to a decomposition of  a Lagrangian cobordism $L$ from $\llinkb$ to $\llinkt$ into elementary pieces specifies a decorated  cobordism between  pointed copies of $\llink_\pm$. We believe that the map $\Phi_L$ agrees with the  \emph{functorial} map of Juh{\'a}sz--Zemke associated to this decorated  cobordism, but  do not prove this here.

\subsection{Effectiveness}
\label{ssec:examples}
In Section \ref{sec:examples}, we give several examples that show that Theorem \ref{thm:main} can be used to obstruct decomposable Lagrangian cobordisms where the classical invariants and smooth topology do not. In particular, we prove the following in Section \ref{ssec:infinite}.

\begin{theorem}\label{thm:examples}
  For each  $g \in\mathbb{Z}_{\geq 0}$, there are  Legendrian knots $\llinkb, 
  \llinkt \subset (\R^3,\xistd)$ such that
  \begin{itemize}
    \item there is a smooth cobordism of genus $g$ in $\R \times \R^3$ 
      between  $\llinkb$ and $\llinkt$,
    \item $\tb(\llinkt) - \tb(\llinkb) = 2 g$ and $\rot (\llinkt) = \rot 
      (\llinkb)$,  
    \item $\lgridhp(\llinkt) = 0$ and $\lgridhp(\llinkb) \neq 0$.
  \end{itemize}
  The last item implies that there is no decomposable Lagrangian cobordism from $\llinkb$ to $\llinkt$.
\end{theorem}

As alluded to  above, Symplectic Field Theory \cite{EGH} also provides effective obstructions to  Lagrangian cobordisms. The most-studied    such SFT obstruction comes from  the  Chekanov--Eliashberg DGA  \cite{Chekanov,Eliashberg}, which assigns to a Legendrian knot $\llink\subset (\R^3,\xistd)$ a differential graded algebra $(\mathcal{A}_{\llink},\partial_\llink)$, which is an invariant of the Legendrian isotopy class of $\llink$, up to stable tame isomorphism. This DGA is said to be \emph{trivial} if it is stable tame isomorphic to a DGA in which the unit is a boundary. Ekholm, Honda, and K{\'a}lm{\'a}n proved  in \cite{EHK} that an exact Lagrangian cobordism from $\llinkb$ to $\llinkt$ induces a DGA morphism \[(\mathcal{A}_{\llinkt},\partial_{\llinkt})\to(\mathcal{A}_{\llinkb},\partial_{\llinkb}).\] Therefore, if the  first DGA is trivial and  the second  is nontrivial then there cannot exist such a cobordism.
It can be difficult to determine whether these DGAs are trivial, meaning that this obstruction can be hard to apply in practice. By contrast, there is a simple algorithm to decide whether the GRID invariants vanish and apply Theorem \ref{thm:main}. 

Another advantage of the GRID invariants is that the elements $\lgridhp$ and $\lgridhm$ are preserved by negative and positive Legendrian stabilization, respectively. This implies, for example, that for   any pair $\llinkb,\llinkt$ of Legendrian knots  as in Theorem \ref{thm:examples},  the GRID invariants also obstruct the existence of a decomposable Lagrangian cobordism from any negative stabilization of $\llinkb$ to any negative stabilization of $\llinkt$. By contrast, the Chekanov--Eliashberg DGA is trivial   for stabilized knots, and therefore cannot obstruct such cobordisms.

We should  point out that there are also examples for which   the  DGA obstructs decomposable Lagrangian cobordisms where the GRID invariants do not (see Section \ref{ssec:dgagrid}).

\subsection{Antecedents}
\label{ssec:antecedents}

As mentioned above, there are a few prior works that use knot Floer homology to obstruct  genus zero Lagrangian cobordisms; such cobordisms are called \emph{Lagrangian concordances}, and are automatically exact. 

In \cite{BalSiv14:KHMLeg}, for instance, Baldwin and Sivek defined  an invariant of Legendrian knots in  arbitrary closed contact 3-manifolds using  monopole knot homology, and showed that it satisfies functoriality with respect to Lagrangian concordances in    symplectizations of such manifolds. They then proved in \cite{BalSiv18:EqInv} that there is an isomorphism between monopole knot homology and knot Floer homology that identifies their Legendrian invariant with the LOSS invariant. This implies that the LOSS invariant is well-behaved with respect to Lagrangian concordances, and, in particular, reproduces Theorem \ref{thm:main} for concordances between knots in the symplectization of $(\R^3,\xistd)$, without the assumption of decomposability.

The other notable  result in this area is due to Golla and Juh{\'a}sz, who proved in \cite{GolJuh18:LOSSConc} that the LOSS invariant satisfies functoriality with respect to \emph{regular} Lagrangian concordances in  Weinstein cobordisms between closed contact 3-manifolds (which include symplectizations). More precisely, they showed 
that the functorial map (of Juh{\'a}sz--Zemke) \[\HFKh(-Y_+,\llinkt)\to\HFKh(-Y_-,\llinkb)\]associated to a decorated regular Lagrangian concordance $L$ in a Weinstein cobordism $W$, \[(W,L):(Y_-,\llinkb)\to(Y_+,\llinkt),\] sends $\lossh(\llinkt)$ to $\lossh(\llinkb)$, for decorations consisting of two parallel arcs that partition the cylinder into  disks. We note that regular cobordisms are exact and, in the symplectization of $(\R^3,\xistd)$, include decomposable cobordisms \cite{ConEtnTos19:SympFillings}; in brief,
\[\{\textrm{decomposable}\}\subseteq \{\textrm{regular}\}\subseteq \{\textrm{exact}\},\]
and it is open whether any of these inclusions are proper. The Golla--Juh{\'a}sz result therefore recovers Theorem \ref{thm:main} for concordances between knots in the symplectization of $(\R^3,\xistd)$, with the \emph{potentially} weaker assumption of regularity.

As noted previously, what most differentiates the results in this paper from those in previous works  is that ours apply to positive genus Lagrangian cobordisms as well as to concordances. The table below summarizes the different settings in which the various knot Floer obstructions to Lagrangian cobordism are known to hold.

 \begin{table}[htbp]
  \begin{tabular}{ l | c c c }
    & & & \\[-1em]
 & \cite{BalSiv14:KHMLeg, BalSiv18:EqInv} & \cite{GolJuh18:LOSSConc} 
    & Present paper\\

    & & & \\[-1em]
    \midrule[\heavyrulewidth]
    & & & \\[-1em]
    For $L$ in symplectization of $(\R^3,\xistd)$  & \cmark & 
    \cmark & \cmark \\
    & & & \\[-1em]
    \midrule
    & & & \\[-1em]
   For $L$ in any symplectization & \cmark & \cmark 
    & \xmark \\
    & & & \\[-1em]
    \midrule
    & & & \\[-1em]
    For $L$ in any Weinstein cobordism & \xmark & \cmark & \xmark \\
    & & & \\[-1em]
    \midrule
    & & & \\[-1em]
     For any decomposable  $\lcob$ & \cmark & \cmark & \cmark \\
    & & & \\[-1em]
    \midrule
    For  any regular  $\lcob$ & \cmark & \cmark & \xmark \\
    & & & \\[-1em]
    \midrule
    & & & \\[-1em]
    For any exact $\lcob$ & \cmark & \xmark & \xmark \\
    & & & \\[-1em]
    \midrule    & & & \\[-1em]
    For $g (\lcob) = 0$ & \cmark & \cmark & \cmark \\
    & & & \\[-1em]
    \midrule
    & & & \\[-1em]
    For  $g (\lcob) > 0$ & \xmark & \xmark & \cmark \\
   

    \midrule[\heavyrulewidth]
  \end{tabular}
  \vspace*{11pt}
  \caption{The settings in which  various    knot Floer obstructions to Lagrangian cobordisms have been established.}
  \label{tab:comparison}
\end{table}

\subsection{Organization}
Section \ref{sec:back} consists of   background material. In Section \ref{sec:proofs}, we prove Theorem \ref{thm:maintilde}, which, as described above, implies Theorem \ref{thm:main}. In Section \ref{sec:examples}, we show via examples that our obstructions are effective, proving Theorem \ref{thm:examples}.

\subsection{Acknowledgements} The authors thank Lenny Ng for his insights on the Chekanov--Eliashberg DGA and, in particular, for computing this DGA for the knot $m(10_{145})$. The authors also thank Caitlin Leverson, Marco Marengon, Lenny Ng, Peter Ozsv{\'a}th, Yu Pan, Josh Sabloff, Steven Sivek,  Zolt{\'a}n Szab{\'o}, and Shea Vela-Vick for helpful conversations. The third author thanks North Carolina State University for their hospitality.

\section{Background}
\label{sec:back}

In this section, we provide some background on Legendrian knots, Lagrangian cobordisms, knot Floer homology, and the GRID invariants. 

\subsection{Legendrian knots and Lagrangian cobordisms}
\label{ssec:backleg}
Let $\xistd=\ker(\alphastd)$  be the standard contact structure on $\R^3$, where \[\alphastd=dz-ydx, \] as in the introduction. Recall that a smooth link $\llink\subset (\R^3,\xistd)$ is called \emph{Legendrian} if \[T_p\llink\subset (\xistd)_p \textrm{ for all } p\in\llink.\]  We will primarily study Legendrian links up to Legendrian isotopy (and will frequently blur the distinction between Legendrian links and Legendrian link types). Furthermore, our Legendrian links will generally be oriented but we will often suppress the orientation from the notation. 

We will  typically represent a   Legendrian link by its \emph{front diagram}, which is its projection to the $xz$-plane, as illustrated in Figure \ref{fig:trefoil}. Note that a Legendrian link is completely determined by its front diagram; in particular, the crossing information is encoded in the slopes of the strands in the diagram (strands with more negative slope pass over strands with less negative slope). Front diagrams for Legendrian isotopic links are related by a sequence of Legendrian planar isotopies and Legendrian Reidemeister moves, shown in the first three diagrams of Figure \ref{fig:decomposable}. 

\begin{figure}[htbp]
 \labellist
  \tiny\hair 2pt
  \pinlabel $x$ at 56 3
    \pinlabel $z$ at 3 55
   \endlabellist
  \includegraphics[width=2.8cm]{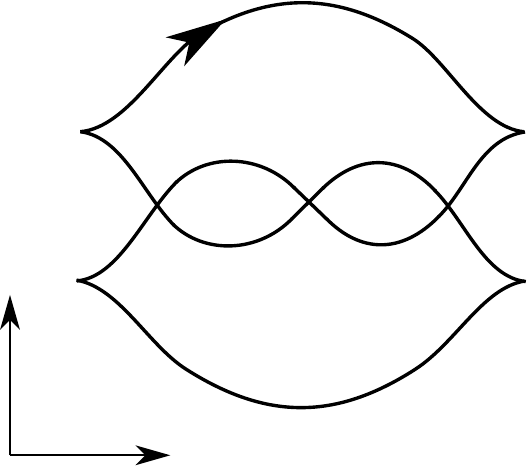}
  \caption{A front diagram for a Legendrian representative of the right-handed trefoil with $(tb,r)=(1,0)$. The positive $y$-axis points into the page.}
  \label{fig:trefoil}
\end{figure}

\begin{figure}[htbp]
 \labellist
  \large\hair 2pt
  \pinlabel $\varnothing$ at 55 24
  \tiny
   \pinlabel $1$ at 171 369
   \pinlabel $2$ at 171 285
   \pinlabel $3$ at 171 201
   \pinlabel $4$ at 171 116
   \pinlabel $5$ at 171 33
    \endlabellist
  \includegraphics[width=5.5cm]{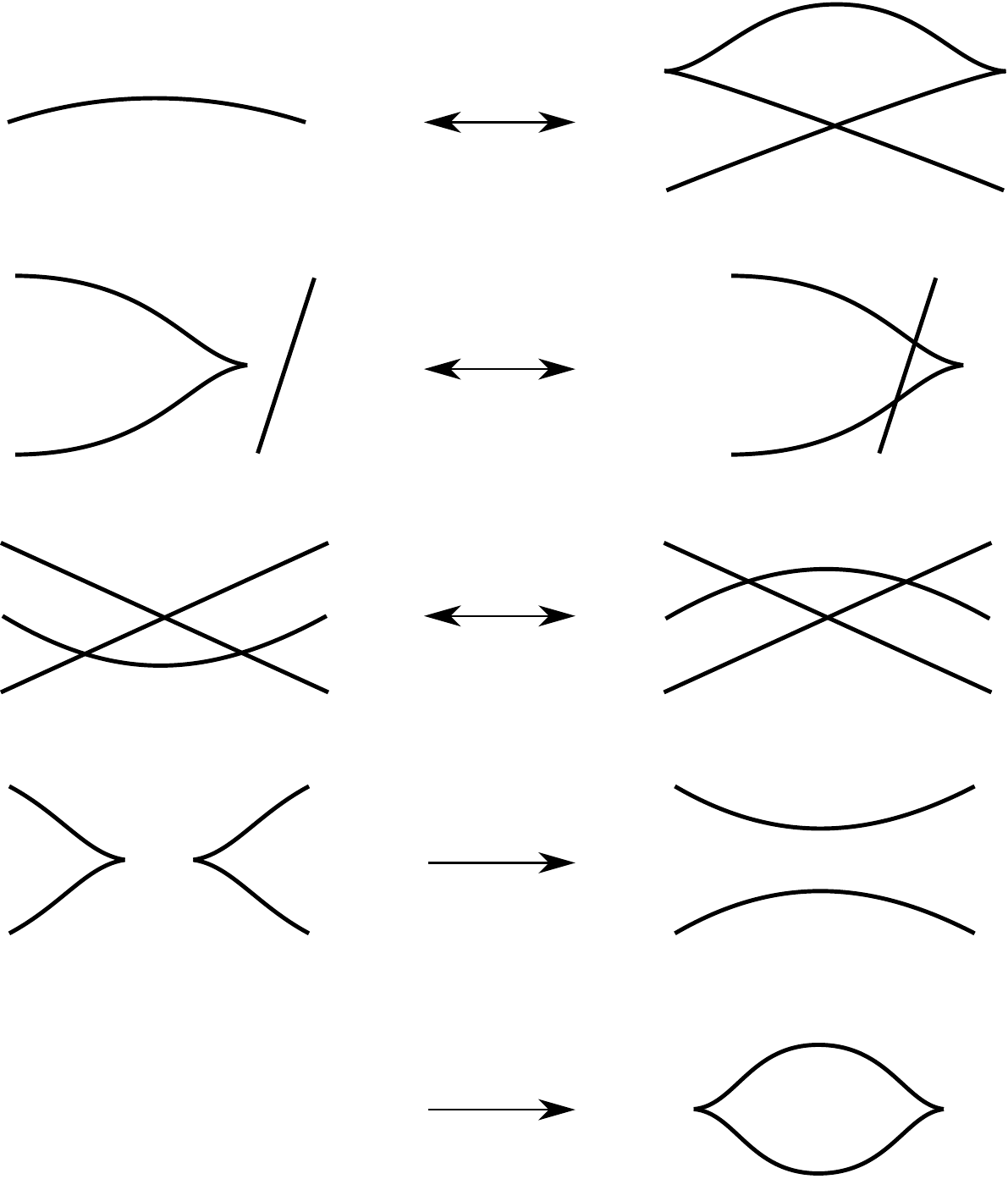}
  \caption{The moves on front diagrams  corresponding to  elementary  cobordisms; horizontal and vertical reflections of these moves are also allowed. Apart from these moves, there are also planar isotopies that preserve left and right cusps. Moves 1-3 are the Legendrian Reidemeister moves; move 4 is called a pinch; move 5 is a birth. Note that moves 4 and 5 are directed.}
  \label{fig:decomposable}
\end{figure}
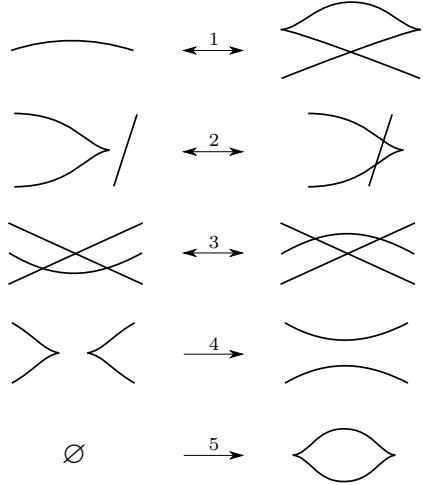

There are two \emph{classical} Legendrian isotopy class invariants: the Thurston--Bennequin number $\tb$ and  rotation number $\rot$. These can be computed from a front diagram by \[
  \tb = \mathrm{wr}- \frac{1}{2}(c_+ + c_-) \textrm{ and } \rot = \frac{1}{2}(c_- - c_+),  
\]
where $\mathrm{wr}$ denotes the writhe of the diagram, and $c_+$ and $c_-$ denote the number of upward and downward pointing cusps in the oriented diagram. 
Two important operations on Legendrian isotopy classes are positive and negative \emph{Legendrian stabilization}. These operations are defined locally in terms of front diagrams  as  in Figure \ref{fig:stab}. In particular, given a Legendrian link $L$, its positive and negative stabilizations $S_+(L)$ and $S_-(L)$ are obtained by adding downward and upward pointing cusps, respectively, as shown in the figure. Note that \begin{equation}\label{eqn:tbrstab}\tb(S_\pm(L)) = \tb(L)-1 \textrm{ and } \rot(S_\pm(L)) =\rot(L)\pm 1. \end{equation}

\begin{figure}[htbp]
 \labellist
  \tiny\hair 2pt
  \pinlabel $+$ at 105 68
    \pinlabel $-$ at 105 29
       \pinlabel $L$ at -10 48
          \pinlabel ${S_{+}}(L)$ at 262 81
             \pinlabel $S_-(L)$ at 262 17
   \endlabellist
  \includegraphics[width=5.3cm]{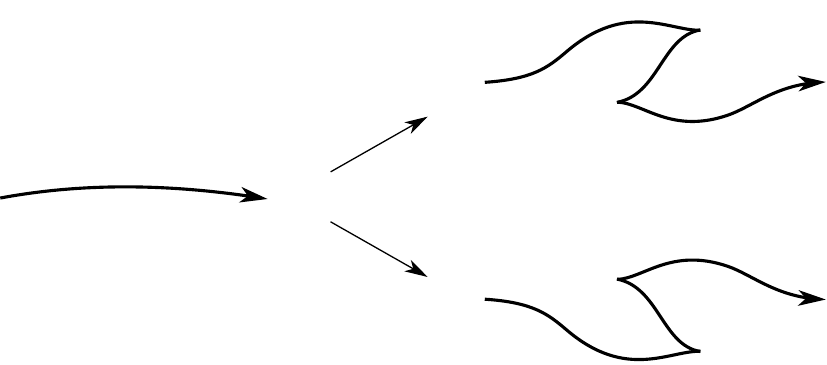}
  \caption{The positive and negative Legendrian stabilizations of a Legendrian link.}
  \label{fig:stab}
\end{figure}
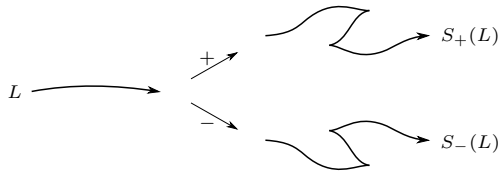

Recall that the \emph{symplectization} of $(\R^3,\xistd)$ is the symplectic 4-manifold  \[(\R_t\times\R^3,d(e^t\alphastd)),\] and that an embedded surface $L$ in the symplectization is called \emph{Lagrangian} if \[d(e^t\alphastd)|_L\equiv 0.\] Suppose $\llinkb,\llinkt$ are two Legendrian links in $(\R^3,\xistd)$.  A \emph{Lagrangian cobordism from $\llinkb$ to $\llinkt$} is an embedded Lagrangian surface $L$ in the symplectization  such that 
\begin{align*}
L\cap ((-\infty,-T)\times \R^3) &= (-\infty,-T)\times\llinkb,\\
L\cap ((T, \infty)\times \R^3) &= (T, \infty)\times\llinkt
\end{align*}
for some $T>0$. This Lagrangian is said to be \emph{exact} if there exists a function $f:L\to \R$ that is constant on the  cylindrical ends and satisfies \[(e^t\alphastd)|_L=df.\] A Lagrangian cobordism of genus zero is called a \emph{Lagrangian concordance}, and is automatically exact.
As mentioned in the introduction, Chaintraine proved in \cite{Cha10:LagConc} that the existence of  a Lagrangian cobordism $L$ from $\llinkb$ to $\llinkt$ implies that \begin{equation}\label{eqn:classicalobstruction}\tb(\llinkt)-\tb(\llinkb) = -\chi(L) \textrm{ and } \rot(\llinkt) = \rot(\llinkb).\footnote{He proved this for cobordisms between Legendrian \emph{knots}, but the proof extends immediately to links.}\end{equation} In particular, Lagrangian cobordisms (even concordances  \cite{chantrainesymmetric}) are  directed. 

By work of Bourgeois, Sabloff, and Traynor \cite{BST}, Chantraine \cite{Cha10:LagConc}, Dimitroglou Rizell \cite{DR}, and Ekholm, Honda, and K{\'a}lm{\'a}n \cite{EHK}, there exists an elementary exact Lagrangian cobordism from $\llinkb$ to $\llinkt$ whenever $\llinkt$ is obtained from $\llinkb$  via Legendrian isotopy, a pinch, or a birth, as illustrated in Figure \ref{fig:decomposable}. Note that for a pinch, it is $\llinkb$ that in fact looks as if it has been obtained from pinching $\llinkt$. Topologically, these  elementary cobordisms  are    annuli, saddles, and cups, respectively. Any composition of elementary cobordisms yields an exact Lagrangian cobordism, and an exact Lagrangian cobordism is called \emph{decomposable} if it is  isotopic through exact Lagrangians to such a composition \cite{Cha12:DecompLag}. As mentioned in the introduction, it is open whether every exact Lagrangian cobordism is decomposable.

\subsection{Knot Floer homology and the  GRID invariants}
\label{ssec:hfkgrid} We begin by reviewing  the grid diagram formulation of knot Floer homology, following the conventions in \cite{OzsStiSza15:GHBook}. See also \cite{ManOzsSar09:GH, ManOzsSza07:GHComb}.

A \emph{grid diagram} $\gd$ is an $n\times n$ grid of squares together with sets \[
\gdo = \{O_1,\dots,O_n\} \textrm{ and }
\gdx = \{X_1,\dots,X_n\}\]
of  markings in the squares    such that each row and column of $\gd$ contains exactly one $O$ marking and one $X$ marking (we   omit the subscripts indexing these markings when convienient);  $n$ is called the \emph{grid number} of $\gd$. We  will think of $\gd$ as  a torus by identifying its top and bottom sides and its left and right sides in the standard way, so that the horizontal grid lines become horizontal circles and the vertical grid lines become  vertical circles, as indicated in Figure \ref{fig:example}.

A grid diagram specifies an oriented link in $\R^3$, obtained as  the  union of  vertical segments from the $X$s to the $O$s in each column with horizontal segments from the $O$s to the $X$s in each row, such that vertical segments pass over horizontal ones, as shown in Figure \ref{fig:example}. 
Conversely, every oriented link in $\R^3$  can be represented by a grid diagram in this way.

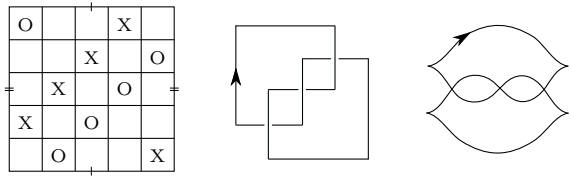
\begin{figure}[htbp]
 \labellist
  \tiny\hair 2pt
  \pinlabel $\mathrm{X}$ at 109 139
    \pinlabel $\mathrm{O}$ at 109 80
\pinlabel $\mathrm{X}$ at 49 80
\pinlabel $\mathrm{O}$ at 49 19
\pinlabel $\mathrm{X}$ at 139 19
\pinlabel $\mathrm{O}$ at 139 109
\pinlabel $\mathrm{X}$ at 79 109
\pinlabel $\mathrm{O}$ at 79 49
\pinlabel $\mathrm{X}$ at 19 49
\pinlabel $\mathrm{O}$ at 19 139
   \endlabellist
  \includegraphics[width=7.5cm]{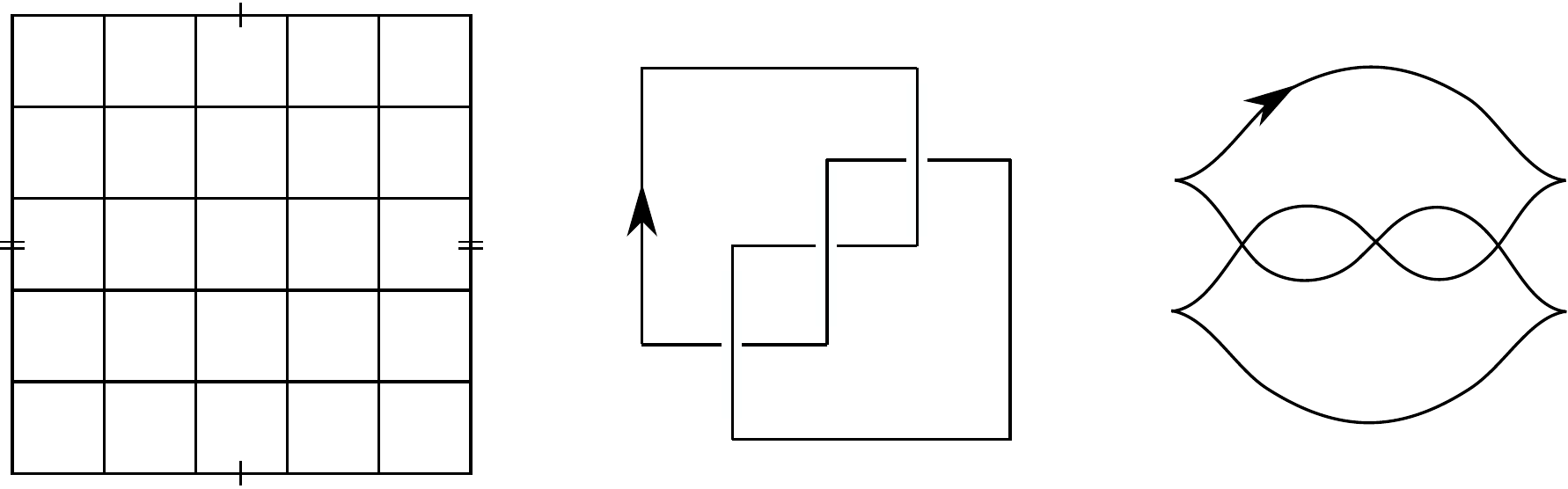}
  \caption{A  grid diagram $\gd$ for the right-handed trefoil $L$, and the corresponding front diagram for a Legendrian representative $\llink$ of $m(L)$, obtained by changing all crossings in the link diagram and rotating $45$ degrees clockwise. }
  \label{fig:example}
\end{figure}

Suppose $\gd$ is a grid diagram as above, representing an oriented link $L$. (The use of $L$ for links will only occur in this subsection, and hence should not cause confusion with Lagrangian cobordisms.)  Let \[\alpha = \{\alpha_1,\dots,\alpha_n\}\textrm{ and }
\beta = \{\beta_1,\dots,\beta_n\}
\] denote the vertical and horizontal circles of $\gd$, respectively.
The \emph{minus} flavor of the grid chain complex, \[(\GCm(\gd),\partial^-),\] is generated by one-to-one correspondences between the vertical and horizontal circles. Equivalently, a generator is a set of $n$ intersection points between these circles where each intersection point in the set belongs to exactly one $\alpha$ circle and  one $\beta$ circle. Letting $\SS(\gd)$ denote the set of   generators, $\GCm(\gd)$ is defined to be the free $\F[U_1,\dots,U_n]$-module generated by the elements of $\SS(\gd)$, where each $U_i$ is a formal variable corresponding to the marking $O_i$ and $\F$ is the 2-element field.

Given $\bx,\by\in\SS(\gd)$, let $\Rect_\gd(\bx,\by)$ be the set of   rectangles in  $\gd$ with the following properties. $\Rect_\gd(\bx,\by)$ is empty unless $\bx$ and $\by$ coincide in exactly $n-2$ intersection points. An element $r\in\Rect_\gd(\bx,\by)$ is an embedded rectangle in the toroidal grid  whose edges are arcs contained in the vertical and horizontal circles, and whose four corners are   points in $\bx\cup\by$. Moreover, we require that, with respect to the induced orientation on $\partial r$, every vertical edge of $r$ is directed from a point in $\by$ to a point in $\bx$, and vice versa for  horizontal edges; that is, \[\partial(\partial_\alpha(r))=\bx-\by\textrm{ and }\partial(\partial_\beta(r))=\by-\bx.\]
(The astute reader may have noticed that this does not seem to line up with the usual convention in Lagrangian Floer homology, but we are in fact computing Heegaard Floer homology for $(-\mathbb{T}^2, \alpha, \beta)$.)
If $\Rect_\gd(\bx,\by)$ is non-empty then it contains exactly two rectangles, as illustrated in Figure \ref{fig:rect}. Let $\Recto_\gd(\bx,\by)$ denote the subset consisting of $r\in\Rect_\gd(\bx,\by)$  with \[r\cap \gdx=\mathrm{Int}(r)\cap \bx=\emptyset.\] The differential \[\partial^-_\gd:\GCm(\gd)\to \GCm(\gd)\] is the $\F[U_1,\dots,U_n]$-module endomorphism defined on $\SS(\gd)$ by \[\partial^-_\gd(\bx) = \sum_{\by\in\SS(\gd)}\,\sum_{r\in\Recto_\gd(\bx,\by)}U_1^{O_1(r)}\cdots U_n^{O_n(r)}\cdot \by,\] where $O_i(r)$ denotes the number of times the marking $O_i$ appears in $r$.

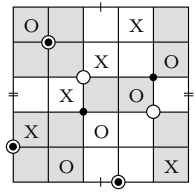
\begin{figure}[htbp]
 \labellist
  \tiny\hair 2pt
  \pinlabel $\mathrm{X}$ at 112 141
    \pinlabel $\mathrm{O}$ at 112 81
\pinlabel $\mathrm{X}$ at 52 81
\pinlabel $\mathrm{O}$ at 52 20
\pinlabel $\mathrm{X}$ at 142 20
\pinlabel $\mathrm{O}$ at 142 110
\pinlabel $\mathrm{X}$ at 82 110
\pinlabel $\mathrm{O}$ at 82 50
\pinlabel $\mathrm{X}$ at 22 50
\pinlabel $\mathrm{O}$ at 22 141
   \endlabellist
  \includegraphics[width=2.5cm]{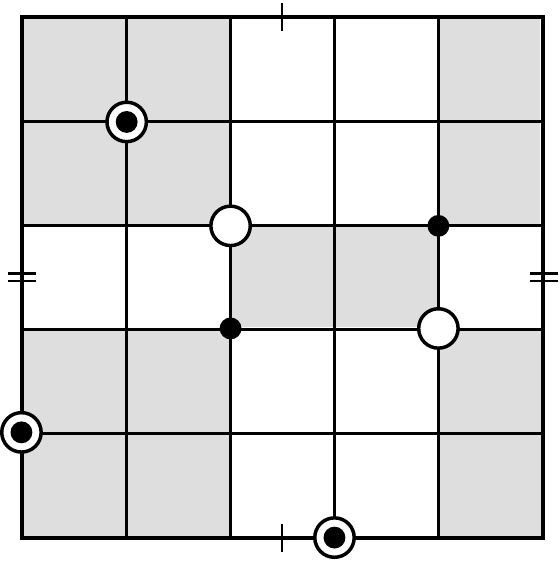}
  \caption{A grid diagram $\gd$. The generator $\bx$ comprises the black intersection points while $\by$ comprises the white  intersection points. $\Rect_\gd(\bx,\by)$ contains the two shaded rectangles shown on this torus, while $\Recto_\gd(\bx,\by)$ contains only the smaller of the two.}
  \label{fig:rect}
\end{figure}

This  complex is equipped with two gradings, the \emph{Maslov} grading and the \emph{Alexander}, defined as follows. Consider the partial ordering on points in $\R^2$ given by \[(p_1,p_2)<(q_1,q_2)\] if $p_1<q_1$ and $p_2<q_2$. Given two sets $P$ and $Q$ consisting of finitely many points in $\R^2$, let \[\mathcal{I}(P,Q) = \#\{(p,q)\in P\times Q\mid p<q\}.\] We symmetrize this quantity by defining \[\mathcal{J}(P,Q) = \frac{\mathcal{I}(P,Q) +\mathcal{I}(Q,P) }{2}.\] A generator $\bx\in\SS(\gd)$ can be viewed as a finite set of points in $\R^2$, as can the marking sets $\gdx$ and $\gdo$. It therefore makes sense to  define \begin{align}
\label{eqn:absmas}M_\gdo(\bx) &= \mathcal{J}(\bx,\bx)-2\mathcal{J}(\bx,\gdo)+\mathcal{J}(\gdo,\gdo)+1,\\
\label{eqn:absalex}M_\gdx(\bx) &= \mathcal{J}(\bx,\bx)-2\mathcal{J}(\bx,\gdx)+\mathcal{J}(\gdx,\gdx)+1.
\end{align}
The Maslov and Alexander gradings of a generator $\bx$ are then given by 
\begin{align*}M(\bx) &= M_\gdo(\bx),\\
A(\bx) &= \frac{1}{2}(M_\gdo(\bx)-M_\gdx(\bx))-\Big( \frac{n-|L|}{2}\Big),
\end{align*}
where $|L|$ is the number of components of the link $L$. It follows that for
$\bx,\by\in\SS(\gd)$ and $r\in\Rect_\gd(\bx,\by)$, the \emph{relative} Maslov and Alexander gradings of these generators are given by \begin{align}
\label{eqn:maslov}M(\bx)-M(\by) &= 1-2\#(r\cap \gdo) + 2\#(\mathrm{Int}(r)\cap\bx),\\
\label{eqn:alexander}A(\bx)-A(\by) &= \#(r\cap \gdx) -\#(r\cap\gdo).
\end{align}
These gradings are extended to gradings on the   complex $\GCm(\gd)$ by the rule that multiplication by any of the $U_i$ lowers Maslov grading by $2$ and Alexander grading by $1$. Note that the differential $\partial^-_\gd$ lowers the Maslov grading by 1 and preserves the Alexander grading.

The \emph{grid homology} of the grid diagram $\gd$ is the Maslov--Alexander bigraded $\F[U_1,\dots,U_n]$-module denoted by \[\GHm(\gd) = H_*(\GCm(\gd),\partial^-_\gd).\] The grid diagram $\gd$ is actually a multi-pointed Heegaard diagram for the link $L\subset S^3$, and the grid chain complex agrees with the \emph{minus} version of the corresponding knot Floer chain complex. Therefore, \[\GHm(\gd) \cong \HFKm(S^3,L).\] Suppose  $L$ has $\ell$ components. Label the markings so that $O_1,\dots,O_\ell$ belong to  the $\ell$ different components of $L$. Setting the corresponding $U_i$ equal to zero on the chain level results in a chain complex \[(\GCh(\gd),\widehat\partial_\gd)=(\GCm(\gd)/(U_1=\dots=U_\ell=0),\partial^-_\gd)\] whose  homology agrees with the \emph{hat} flavor of knot Floer homology, \[\GHh(\gd) = H_*(\GCh(\gd),\widehat\partial_\gd) \cong \HFKh(S^3,L).\] Setting \emph{all} of the $U_i$ to zero yields the \emph{tilde} version of the grid  complex, \[(\GCt(\gd),\widetilde\partial_\gd)=(\GCm(\gd)/(U_1=\dots=U_n=0),\partial^-_\gd),\] whose   homology  agrees with the \emph{tilde} version of knot Floer homology, and is related to the \emph{hat} flavor by \begin{equation}\label{eqn:tildehat}\GHt(\gd) = H_*(\GCt(\gd),\widetilde\partial_\gd) \cong \HFKh(S^3,L)\otimes V^{\otimes n-\ell},\end{equation} where \[V= \F_{0,0} \oplus \F_{-1,-1}\]   is the two-dimensional vector space supported in   Maslov--Alexander bigradings $(0,0)$ and $(-1,-1)$. Finally, the quotient map \[j:\GCh(\gd) \to \GCt(\gd)\] induces an injection \[j_*:\GHh(\gd) \to \GHt(\gd)\]  on homology \cite{NgOzsThu08:GRIDEffective}.

Suppose $\gd$ is a grid diagram representing $L$. By changing all crossings in the associated link diagram, rotating $45$ degrees clockwise, smoothing the top and bottom pointing corners, and turning the left and right pointing corners into cusps, we obtain a front diagram for a Legendrian representative $\llink$ of $m(L)$, as indicated in Figure \ref{fig:example}. We say that a grid diagram $\gd$ \emph{represents} a Legendrian link $\llink$ if the front diagram obtained from $\gd$ in the manner  above is isotopic to the front diagram for $\llink$. Every Legendrian link in $(\R^3,\xistd)$ can be represented by a grid diagram in this way. 

Suppose $\gd$ represents the smooth link $L$ and a Legendrian representative $\llink$ of $m(L)$ as above. As shown in \cite{OzsSzaThu08:GRID}, there are two canonical cycles \[\xp(\gd),\xm(\gd)\in\GCm(\gd)\] consisting of the intersection points to the immediate upper right and   lower left, respectively, of the markings in $\gdx$. These  two generators give rise to cycles in the \emph{hat} and \emph{tilde} complexes as well, which we denote in the same way. The Maslov and Alexander gradings of these cycles are given by \begin{align}
\label{eqn:masgrid}M(\xpm(\gd)) &=\tb(\llink)\mp\rot(\llink)+1 ,\\
\label{eqn:alexgrid}A(\xpm(\gd)) &= \frac{1}{2}(\tb(\llink)\mp\rot(\llink)+|\llink|),
\end{align} where $|\llink|$ is the number of components of $\llink$. In particular, the gradings of the two generators recover $\tb(\llink)$ and $\rot(\llink)$.
The \emph{hat} and \emph{tilde} versions of the GRID invariants of the Legendrian link $\llink$ are then defined   \cite{OzsSzaThu08:GRID} by \[\lgridhpm(\llink)=\lgridhpm(\gd):=[\xpm(\gd)]\in \GHh(\gd) \cong \HFKh(S^3,L) \cong \HFKh(-S^3, \llink)\] and \[\lgridtpm(\gd):=[\xpm(\gd)]\in \GHt(\gd).\] In particular, \[\lgridtpm(\gd) = j_*(\lgridhpm(\gd)),\] which implies that \begin{equation}\label{eqn:zeroiff}\lgridhpm(\llink)=\lgridhpm(\gd)= 0 \textrm{ iff }\lgridtpm(\gd)= 0\end{equation} since $j_*$ is injective.\footnote{Unlike for the \emph{hat} flavor of the GRID invariants, we do not denote $\lgridtpm(\gd)$ by  $\lgridtpm(\llink)$ since  these classes (and the group $\GHt(\gd)$) depend not only on the Legendrian link $\llink$ but also on the grid number $n$, as in \eqref{eqn:tildehat}.}

Ozsv{\'a}th, Szab{\'o}, and Thurston proved  that  $\lgridhpm(\gd)$ are invariants of the Legendrian isotopy class of $\llink$. Specifically, if $\gd_0$ and $\gd_1$ are grid diagrams representing Legendrian isotopic links then there is an isomorphism  \cite[Theorem 1.1]{OzsSzaThu08:GRID} \[\GHh(\gd_0)\to\GHh(\gd_1)\] of Maslov--Alexander bidegree $(0,0)$ that sends $\lgridhpm(\gd_0)$ to $\lgridhpm(\gd_1)$. This  map is defined combinatorially, in terms of chain maps on the grid  complex associated to grid diagram versions of the Legendrian Reidemeister moves. Their argument also gives rise to the following statement for the \emph{tilde} flavor
 of the GRID invariants.

\begin{proposition}
\label{prop:isotopy}
If $\gd_0$ and $\gd_1$ are grid diagrams representing  Legendrian isotopic links then there is a homomorphism \[\GHt(\gd_0)\to \GHt(\gd_1)\]  that sends $\lgridtpm(\gd_0)$ to $\lgridtpm(\gd_1)$. This map has Maslov--Alexander bidegree $(0,0)$.
\end{proposition}

Note that the homomorphism above may not be an isomorphism. The GRID invariants also behave as follows under stabilization  \cite[Theorem 1.3]{OzsSzaThu08:GRID}.

\begin{proposition} 
\label{prop:stab}Suppose $\gd$ is a grid representative of a Legendrian link $\llink$, and that $\gd_\pm$ are grid representatives of the positive and negative Legendrian stabilizations $S_\pm(\llink)$, respectively. Then \[\lgridhp(\gd_+) = \lgridhm(\gd_-)=0,\] and \[\lgridhp(\gd_-) = 0 \textrm{ iff } \lgridhp(\gd)=0 \textrm{ and }
\lgridhm(\gd_+) = 0\textrm{ iff }\lgridhm(\gd) = 0 .\] The analogous statement holds for the \emph{tilde}  invariants, by \eqref{eqn:zeroiff}.
\end{proposition}

\section{Proofs of main results}
\label{sec:proofs}


To define the map $\Phit_L$ in Theorem \ref{thm:maintilde} associated to a decomposable Lagrangian cobordism $L$, we first define maps  associated to Legendrian isotopies, pinches, and births, as discussed in the introduction. The maps associated to Legendrian isotopies were   defined previously by Ozsv{\'a}th, Szab{\'o}, and Thurston  in  \cite{OzsSzaThu08:GRID}, and are described   in Proposition \ref{prop:isotopy}, so we will restrict our attention below to the maps associated to pinches and births.

\subsection{Pinches}

\begin{proposition}
\label{prop:pinch} Suppose $\llinkt$ is obtained from $\llinkb$ via a pinch move. For any  grid diagrams   $\gd_+$ and $\gd_-$ representing $\llinkt$ and $\llinkb$, respectively, there is a homomorphism \[\Phit:\GHt(\gd_+)\to\GHt(\gd_-)\] 
that sends $\lgridtpm(\gd_+)$ to $\lgridtpm(\gd_-)$. This map has Maslov--Alexander bidegree \[\begin{cases}
(-1,0),&\textrm{if } |\llinkb|=|\llinkt|+1,\\
(-1,-1),&\textrm{if } |\llinkb|=|\llinkt|-1,
\end{cases}\] where $|\llink_\pm|$ is the number of components of $\llink_\pm$.
\end{proposition}

\begin{proof}
By Proposition \ref{prop:isotopy}, it suffices to show that there exist \emph{some} grid diagrams $\gd_+$ and $\gd_-$ representing  links  Legendrian isotopic to $\llinkt$ and $\llinkb$, respectively,  for which the conclusions of Proposition \ref{prop:pinch} hold. For this, note that there are grid diagrams $\gd_\pm$ representing $\llink_\pm$ that are identical except for the positions of two markings in adjacent rows, as  shown in Figure \ref{fig:saddle4}. 
Since $L$ is an \emph{oriented} cobordism, these two special markings must either both be $X$s, which we refer to as Case I, or   both be $O$s, which we refer to as Case II. 

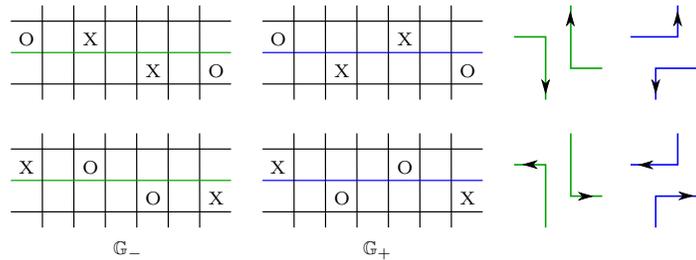
\begin{figure}[htbp]
\labellist
  \tiny

  \pinlabel $\mathrm{O}$ at 7.5 91
\pinlabel $\mathrm{X}$ at 38 91
  \pinlabel $\mathrm{X}$ at 68 76
\pinlabel $\mathrm{O}$ at 98 76

  \pinlabel $\mathrm{O}$ at 127.5 91
\pinlabel $\mathrm{X}$ at 188 91
  \pinlabel $\mathrm{X}$ at 158 76
\pinlabel $\mathrm{O}$ at 218 76

  \pinlabel $\mathrm{X}$ at 7.5 30
\pinlabel $\mathrm{O}$ at 38 30
  \pinlabel $\mathrm{O}$ at 68 15
\pinlabel $\mathrm{X}$ at 98 15

  \pinlabel $\mathrm{X}$ at 127.5 30
\pinlabel $\mathrm{O}$ at 188 30
  \pinlabel $\mathrm{O}$ at 158 15
\pinlabel $\mathrm{X}$ at 218 15

\pinlabel $\gd_-$ at 56 -10
\pinlabel $\gd_+$ at 175 -10
     \endlabellist
  \includegraphics[width=9.2cm]{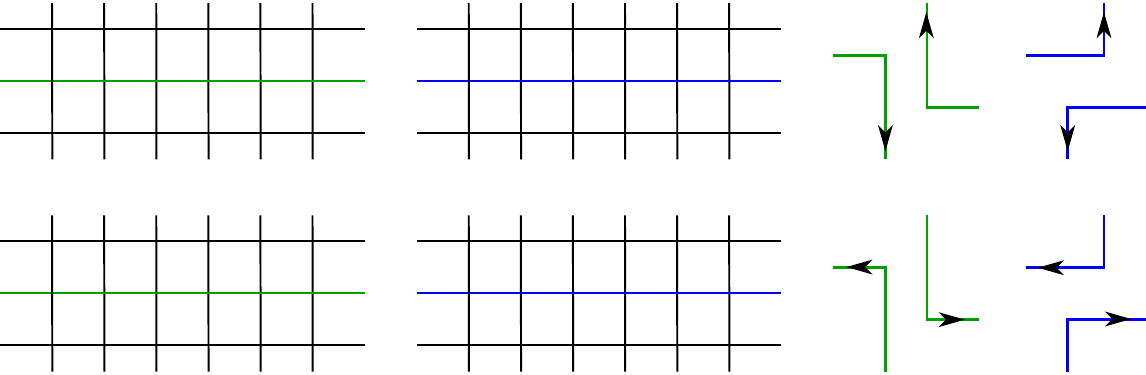}
  \caption{The grid diagrams for $\gd_\pm$ corresponding to a pinch move; Case I  on the top, Case II on the bottom.}
  \label{fig:saddle4}
\end{figure}

We may combine the grid diagrams $\gdb$ and $\gdt$ into a single toroidal diagram, as shown in Figure \ref{fig:saddle3}, which we will refer to as the \emph{combined diagram}. From this perspective, the markings in $\gdx, \gdo$ are fixed and $\gdb$ and $\gdt$ differ in a single horizontal circle. We  denote these differing horizontal circles by $\beta$ and $\gamma$, as shown in Figure \ref{fig:saddle3}. Let $a$ and $b$ be the intersection points of $\beta$ with $\gamma$ shown in  the figure. Below, we define the map $\Phi$ for each of Cases I and II.

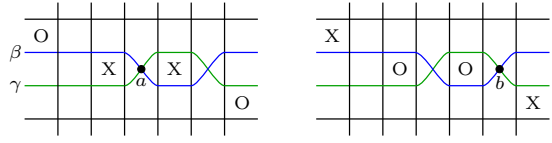
\begin{figure}[htbp]
 \labellist
  \tiny
  \pinlabel $\beta$ at -4 37
  \pinlabel $\gamma$ at -4 22
  \pinlabel $\mathrm{O}$ at 7.5 45
\pinlabel $\mathrm{X}$ at 38 30
  \pinlabel $\mathrm{X}$ at 67.5 30
\pinlabel $a$ at 52.5 24
\pinlabel $\mathrm{O}$ at 98 15

  \pinlabel $\mathrm{X}$ at 138.5 45
\pinlabel $\mathrm{O}$ at 169 30
  \pinlabel $\mathrm{O}$ at 198.5 30
  \pinlabel $b$ at 214 24
\pinlabel $\mathrm{X}$ at 229 15
    \endlabellist
  \includegraphics[width=7cm]{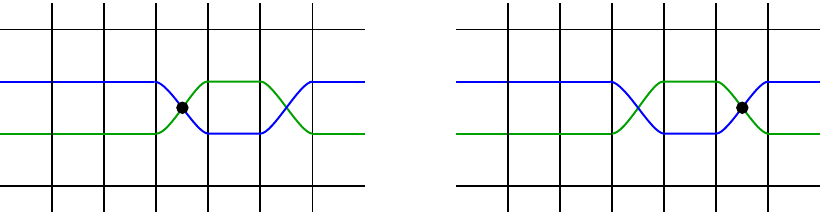}
  \caption{The grid diagrams $\gd_\pm$ combined; Case I on the left, Case II on the right.}
  \label{fig:saddle3}
\end{figure}

\subsubsection{Case I}
For $\bx\in\SS(\gdt)$ and $\by\in\SS(\gdb)$, let $\Pent(\bx,\by)$ be the space of  pentagons in the combined diagram  with the following properties. $\Pent(\bx,\by)$ is empty unless $\bx$ and $\by$ coincide in exactly $n-2$ intersection points, where $n$ is the grid number of $\gd_\pm$. An element $p\in\Pent(\bx,\by)$ is an embedded pentagon in the toroidal diagram whose edges are arcs contained in the vertical and horizontal circles, and whose five corners are points in $\bx\cup\by\cup\{a\}$. We require that, with respect to the induced orientation on $\partial p$, the boundary of this pentagon may be traversed as follows: start at the point in $\bx$ on $\beta$ and proceed along an arc of $\beta$ until arriving at  $a$; next, proceed along an arc of $\gamma$ until arriving at a point in $\by$; next, follow an arc of a vertical circle until arriving at a point in $\bx$; next, proceed along an arc of a horizontal circle until arriving at a point in $\by$; finally, follow an arc of a vertical circle back to the initial point in $\bx$. See Figure \ref{fig:saddle2} for  such pentagons. Let $\Pento(\bx,\by)$ be the subset consisting of  $p\in \Pent(\bx,\by)$  with
\[p\cap \gdo = p\cap \gdx = \mathrm{Int}(p)\cap \bx = \emptyset.\]  Let \[\phi:\GCt(\gdt)\to\GCt(\gdb)\] be the linear map defined on generators by counting such pentagons, \[\phi(\bx) = \sum_{\by\in\SS(\gdb)}\,\sum_{p\in\Pento(\bx,\by)} \by.\]

\begin{lemma}
\label{lem:chainpinch1}
$\phi$ is a chain map.
\end{lemma} 

\begin{proof}
To show that $\phi$ is a chain map, we must prove the equality of coefficients, \[\langle(\widetilde\partial_{\gdb}\circ\phi)(\bx),\by\rangle = \langle (\phi\circ\widetilde\partial_{\gdt})(\bx),\by\rangle,\] for every pair of generators $\bx\in\SS(\gdt)$ and $\by\in\SS(\gdb)$. The coefficients on the left and right    count concatenations of rectangles and pentagons from $\bx$ to $\by$ of the forms $p*r$ and $r*p$, respectively, where $p$ is a pentagon of the sort used to define $\phi$, and $r$ is a rectangle of the sort used to define the differentials. Every domain in the combined diagram that decomposes as  the juxtaposition of  a rectangle and pentagon in this way admits exactly one other such decomposition, exactly  as in the proof of commutation invariance for grid homology \cite[Lemma~3.1]{ManOzsSza07:GHComb}. In particular, the concatenations of pentagons and rectangles contributing to the coefficients above cancel in pairs, proving the lemma.\end{proof}

\begin{lemma}
\label{lem:mapstopinch1}
$\phi$ sends $\xpm(\gdt)$ to $\xpm(\gdb)$.
\end{lemma} 

\begin{proof} 
There is a unique pentagon contributing to each of $\phi(\xp(\gdt))$ and $\phi(\xm(\gdt))$, shown in Figure \ref{fig:saddle2}, that certifies that \[\phi(\xpm(\gdt))=\xpm(\gdb).\]

\begin{figure}[htbp]
 \labellist
  \tiny
  \pinlabel $\mathrm{O}$ at 7.5 45
\pinlabel $\mathrm{X}$ at 38 30
  \pinlabel $\mathrm{X}$ at 67.5 30
\pinlabel $a$ at 52.5 24
\pinlabel $\mathrm{O}$ at 98 15

  \pinlabel $\mathrm{O}$ at 138.5 45
\pinlabel $\mathrm{X}$ at 169 30
  \pinlabel $\mathrm{X}$ at 198.5 30
\pinlabel $\mathrm{O}$ at 229.5 15
    \endlabellist
  \includegraphics[width=7cm]{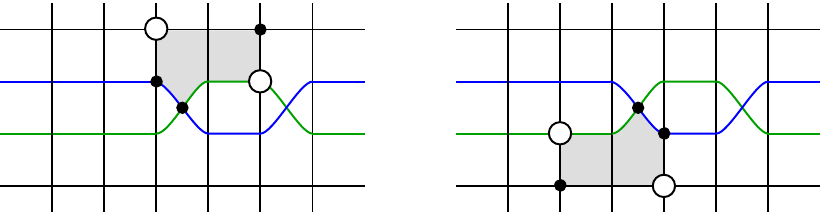}
  \caption{Left, the pentagon certifying that  the map $\phi$ sends $\xp(\gdt)$ in black to $\xp(\gdb)$ in white. Right, the pentagon from $\xm(\gdt)$  to $\xm(\gdb)$.}
  \label{fig:saddle2}
\end{figure}
\end{proof}

\begin{lemma}
\label{lem:homopinch1}
$\phi$ is homogeneous of  Maslov--Alexander bidegree  \[\begin{cases}
(-1,0),&\textrm{if } |\llinkb|=|\llinkt|+1,\\
(-1,-1),&\textrm{if } |\llinkb|=|\llinkt|-1.
\end{cases}\]
\end{lemma} 

\begin{proof}
This is a  straightforward calculation from the definitions of the Maslov and Alexander gradings in \eqref{eqn:absmas} and \eqref{eqn:absalex}, and the map $\phi$; see the proofs of \cite[Lemma 5.3.1]{OzsStiSza15:GHBook} and \cite[Lemma 6.6]{Won17:GHUnorientedSkein}, for example.
\end{proof}

\begin{remark}
If one could show that $\phi$ is homogeneous (say, using the relative grading formulas in \eqref{eqn:maslov} and \eqref{eqn:alexander}), the bidegree of $\phi$ would be  determined by the fact that this map sends $\xp(\gdt)$ to $\xp(\gdb)$, and the lemma would follow immediately from  \eqref{eqn:masgrid} and \eqref{eqn:alexgrid}, together with the facts that  \begin{align*}
\tb(\llinkb) &= \tb(\llinkt)-1, \\
\rot(\llinkb)&=\rot(\llinkt). 
\end{align*} 
\end{remark}

\subsubsection{Case II} For $\bx\in\SS(\gdt)$ and $\by\in\SS(\gdb)$, let $\Tri(\bx,\by)$ be the space of  triangles in the combined diagram  with the following properties. $\Tri(\bx,\by)$ is empty unless $\bx$ and $\by$ coincide in exactly $n-1$ intersection points. An element $p\in\Tri(\bx,\by)$ is an embedded triangle in the torus whose edges are arcs contained in the vertical and horizontal circles, and whose three corners are points in $\bx\cup\by\cup\{b\}$. We require that, with respect to the induced orientation on $\partial p$, the boundary of this triangle may be traversed as follows: start at the point in $\bx$ on  $\beta$ and proceed along an arc of $\beta$ until arriving at  $b$; next, proceed along an arc of $\gamma$ until arriving at a point in $\by$; finally, follow an arc of a vertical circle back to the initial point in $\bx$. See Figure \ref{fig:saddle1} for such triangles. Note that all such triangles  automatically satisfy \[p\cap \gdx = \mathrm{Int}(p)\cap \bx = \emptyset.\] Let $\Trio(\bx,\by)$ be the subset consisting of  $p\in \Tri(\bx,\by)$  with $p\cap\gdo = \emptyset.$  Let \[\phi:\GCt(\gdt)\to\GCt(\gdb)\] be the linear map defined on generators by counting such triangles, \[\phi(\bx) = \sum_{\by\in\SS(\gdb)}\,\sum_{p\in\Trio(\bx,\by)} \by.\]

\begin{lemma}
$\phi$ is a chain map.
\end{lemma}

\begin{proof}
This follows from an argument identical to that in the proof of Lemma \ref{lem:chainpinch1}, except that here we consider canceling concatenations of rectangles with triangles rather than pentagons. See  the proof of \cite[Lemma~3.4]{Won17:GHUnorientedSkein} for details in this case.
\end{proof}

\begin{lemma}
\label{lem:mapstopinch2}
$\phi$ sends $\xpm(\gdt)$ to $\xpm(\gdb)$.
\end{lemma} 

\begin{proof} There is a unique triangle contributing to each of $\phi(\xp(\gdt))$ and $\phi(\xm(\gdt))$, shown in Figure \ref{fig:saddle1}, that certifies that \[\phi(\xpm(\gdt))=\xpm(\gdb).\]

\begin{figure}[htbp]
 \labellist
  \tiny
  \pinlabel $\mathrm{X}$ at 15 45
\pinlabel $\mathrm{O}$ at 45 30
  \pinlabel $\mathrm{O}$ at 74 30
\pinlabel $b$ at 90 24
\pinlabel $\mathrm{X}$ at 105 15

  \pinlabel $\mathrm{X}$ at 169.5 45
\pinlabel $\mathrm{O}$ at 199.5 30
  \pinlabel $\mathrm{O}$ at 229.5 30
\pinlabel $\mathrm{X}$ at 259.5 15
    \endlabellist
  \includegraphics[width=8cm]{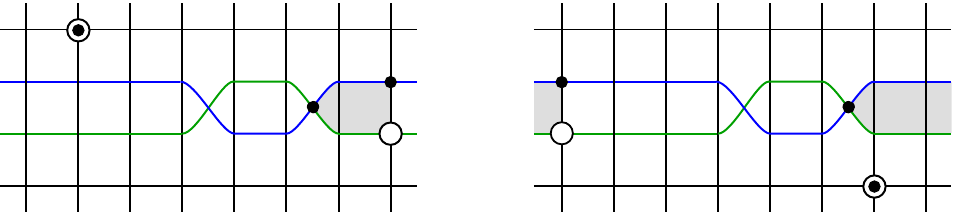}
  \caption{Left, the triangle certifying that  the map $\phi$ sends $\xp(\gdt)$ in black to $\xp(\gdb)$ in white. Right, the triangle from $\xm(\gdt)$  to $\xm(\gdb)$.}
  \label{fig:saddle1}
\end{figure}
\end{proof}

\begin{lemma}
\label{lem:homopinch2}
$\phi$ is homogeneous of  Maslov--Alexander bidegree   \[\begin{cases}
(-1,0),&\textrm{if } |\llinkb|=|\llinkt|+1,\\
(-1,-1),&\textrm{if } |\llinkb|=|\llinkt|-1.
\end{cases}\]
\end{lemma} 
\begin{proof}
As with Lemma \ref{lem:homopinch1}, this is a  straightforward calculation from the definitions of these  gradings in \eqref{eqn:absmas} and \eqref{eqn:absalex}, and the map $\phi$; see the proof of  \cite[Lemma 6.6]{Won17:GHUnorientedSkein} for details.
\end{proof}

The map $\Phit$ induced by  $\phi$ therefore  satisfies the conclusions of Proposition \ref{prop:pinch}.
\end{proof}
\subsection{Births}

\begin{proposition}
\label{prop:birth} Suppose $\llinkt$ is obtained from $\llinkb$ via a birth move. For any   grid diagrams   $\gd_+$ and $\gd_-$ representing $\llinkt$ and $\llinkb$, respectively, there is a homomorphism \[\Phit:\GHt(\gd_+)\to\GHt(\gd_-)\] 
that sends $\lgridtpm(\gd_+)$ to $\lgridtpm(\gd_-)$. This map has  Maslov--Alexander  bidegree $(1,0)$.
\end{proposition}

\begin{proof}
By Proposition \ref{prop:isotopy}, it suffices to show that there exist \emph{some} grid diagrams $\gd_+$ and $\gd_-$ representing  links  Legendrian isotopic to $\llinkt$ and $\llinkb$, respectively,  for which the conclusions of Proposition \ref{prop:birth} hold. For this, let $\gdb$ be any grid diagram representing $\llinkb$, with marking sets $\gdx,\gdo$. Fix a marking $X_1\in\gdx$. Let $\gdt$ be the grid diagram obtained from $\gdb$ by inserting two rows and two columns to the immediate bottom right of $X_1$, with four new markings $X_2,X_3,O_2,O_3$, as shown  in Figure \ref{fig:birth0}.  Let $a$ and $b$ be the intersection points between the new vertical and horizontal circles indicated in the figure. Note that  $\gdt$ represents the disjoint union of $\llinkb$ with the $tb=-1$ Legendrian unknot, which is Legendrian isotopic to $\llinkt$.

\begin{figure}[htbp]
 \labellist
  \tiny
  \pinlabel $\alpha_3$ at 16 76
  \pinlabel $\alpha_1$ at 75 76
  \pinlabel $\alpha_2$ at 90 76
  \pinlabel $\alpha_3$ at 105 76
  
   \pinlabel $\beta_3$ at 37 57
  \pinlabel $\beta_1$ at 118 57
  \pinlabel $\beta_2$ at 118 42
  \pinlabel $\beta_3$ at 118 27
    \pinlabel $a$ at 71.5 54
      \pinlabel $b$ at 87 38.5
\pinlabel $\mathrm{X}_1$ at 7.5 63.5
\pinlabel $\mathrm{X}_1$ at 67.5 63.5
\pinlabel $\mathrm{O}_2$ at 82.5 49
\pinlabel $\mathrm{X}_2$ at 82.5 33.5
\pinlabel $\mathrm{O}_3$ at 98 33.5
\pinlabel $\mathrm{X}_3$ at 98 49
    \endlabellist
  \includegraphics[width=3.5cm]{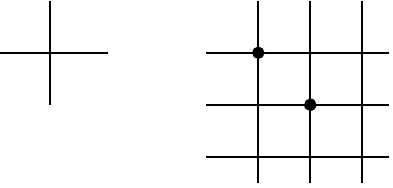}
  \caption{Left, part of a grid diagram $\gdb$ for $\llinkb$. Right, the corresponding part of the  grid diagram  $\gdt$ for $\llinkt$.}
  \label{fig:birth0}
\end{figure}

The generating set $\SS(\gdt)$ can be expressed as a disjoint union, \[\SS(\gdt) =  \AB  \cup \AN \cup \NB  \cup \NN 
,\] where
  \begin{itemize}
  \item $\AB$ consists of  $\bx \in \SS (\gdt)$ with $a, b 
    \in \bx$,
  \item $\AN $ consists of  $\bx \in \SS (\gdt)$ with $a \in 
    \bx$ but $b \notin \bx$,
  \item $\NB$ consists of  $\bx \in \SS (\gdt)$ with $a 
    \notin \bx$ but $b \in \bx$, 
  \item $\NN$ consists of  $\bx \in \SS (\gdt)$ with $a\notin\bx$ and $b 
    \notin \bx$.
\end{itemize}
This induces a decomposition of the vector space $\GCt(\gdt)$ as a direct sum,
\[\GCt(\gdt) =  \ABt \oplus \ANt  \oplus \NBt \oplus \NNt ,\]
where these  summands are the vector spaces generated by the corresponding subsets of $\SS(\gdt)$. Note that we have a sequence of subcomplexes, \[\NBt\subset \ABt\oplus\NBt\subset \GCt(\gdt).\] This follows immediately from the observation that  any rectangle either starting at $b$ or terminating at $a$ must pass through one of the new markings or $X_1$ (and therefore does not contribute to the differential). Let \[(\ABt,\widetilde\partial_{\AB})\] be the quotient complex of $\ABt\oplus\NBt$ by $\NBt$. In other words, for $\bx,\by\in\AB$, the coefficient \[\langle\widetilde\partial_{\AB}(\bx),\by\rangle=\langle\widetilde\partial_{\gdt}(\bx),\by\rangle\] counts the number of rectangles in $\Recto_{\gdt}(\bx,\by)$, as usual.

Note that there is a bijection between generators in $\AB$ and generators in $\SS(\gdb)$, given by \[\bx \mapsto \bx\smallsetminus \{a,b\}.\] This  bijection extends linearly to an isomorphism of  chain complexes, \[e:\ABt\to\GCt(\gdb),\] since for $\bx,\by\in\AB$ there is also a natural bijection \[\Rect_{\gdt}(\bx,\by)\to\Rect_{\gdb}(\bx\smallsetminus\{a,b\},\by\smallsetminus\{a,b\}),\] which identifies rectangles  avoiding the $O$ and $X$ markings in $\gdt$ with rectangles avoiding the $O$ and $X$ markings in $\gdb$. Moreover, it follows  readily  from \eqref{eqn:maslov} and \eqref{eqn:alexander}, together with this bijection of rectangles, that $e$ is homogeneous with respect to the Maslov--Alexander bigrading.

For $\bx\in\NB$ and $\by\in \AB$, let \[\Rect_{\AB}(\bx,\by)\subset\Rect_{\gdt}(\bx,\by)\] be the subset  consisting of rectangles $p$ satisfying
\begin{itemize}
  \item $p \cap (\gdo \cup  \{O_2, O_3\})= \{O_2, O_3\}$,
   \item $p \cap (\gdx \cup  \{X_2, X_3\})= \{X_2, X_3\}$,
  \item $\mathrm{Int} (p) \cap \bx = \mathrm{Int} (p) \cap \by = \{b\}$.
\end{itemize}
Let \[\psi:\NBt\to\ABt\] be the linear map defined on generators by counting such rectangles, \[\psi(\bx) = \sum_{\by\in\AB}\,\sum_{p\in\Rect_{\AB}(\bx,\by)} \by.\] Let \[\Pi:\GCt(\gdt)\to \NBt\] be projection onto the summand $\NBt$, and define \[\phi:\GCt(\gdt)\to\GCt(\gdb)\] to be the linear map given  as the composition \[\phi = e\circ \psi\circ\Pi.\]

\begin{lemma}
$\phi$ is a chain map.
\end{lemma} \begin{proof}Since $e$ is a chain map, it suffices to prove that $ \psi\circ\Pi$ is a chain map. Note  that both \[\widetilde\partial_{\AB}\circ(\psi\circ\Pi)\textrm{ and }(\psi\circ\Pi)\circ\widetilde\partial_{\gdt}\] vanish for generators $\bx\notin \AB\cup\NB$. The first  vanishes on such generators $\bx$ because  $\Pi(\bx) = 0.$ The  second vanishes on such generators $\bx$ because  \[(\Pi\circ\widetilde\partial_{\gdt})(\bx)=0.\] Indeed, for every $\by\in\NB$, the coefficient \[\langle \widetilde\partial_{\gdt}(\bx),\by\rangle = 0\] since every rectangle from a generator not containing $b$ to a generator containing $b$ must pass through a marking. To prove that $\phi$ is a chain map, it therefore suffices to prove the equality of coefficients \[\langle(\widetilde\partial_{\AB}\circ\psi\circ\Pi)(\bx),\by\rangle = \langle (\psi\circ\Pi\circ\widetilde\partial_{\gdt})(\bx),\by\rangle\] for every pair  of generators $\bx\in\AB\cup \NB$ and $\by\in\AB$. These coefficients on the left and right    count  concatenations of rectangles from $\bx$ to $\by$ of the forms $p*r$ and $r*p$, respectively, where $p$ is a rectangle of the sort used to define  $\psi$, and $r$ is a rectangle of the sort used to define the differentials. For $\bx\in\NB$, every domain in $\gd_+$ that decomposes as  the juxtaposition of  the form $p*r$ admits exactly one other decomposition into rectangles $p$ and $r$, of the form $r*p$, exactly as in the proof that the grid  differential squares to zero \cite[Proposition 2.10]{ManOzsSza07:GHComb}. There are additional domains that decompose as the juxtaposition of the form $r * p$, but these cancel in pairs as well. In particular, the concatenations of  rectangles contributing to the coefficients above cancel in pairs, proving the lemma in this case. For $\bx\in\AB$, the first coefficient is  zero since $\Pi(\bx)=0,$ and there are exactly two concatenations of the form $r*p$ contributing to the second coefficient. These two cancelling concatenations correspond to vertical and horizontal annular domains of width 2, as shown in Figure \ref{fig:birth1}.  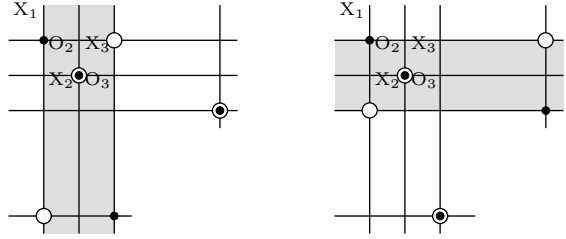
\begin{figure}[htbp]
 \labellist
  \tiny
  \pinlabel $\mathrm{X}_1$ at 7.5 97
\pinlabel $\mathrm{O}_2$ at 22.5 82.5
\pinlabel $\mathrm{X}_2$ at 22.5 67
\pinlabel $\mathrm{O}_3$ at 38 67
\pinlabel $\mathrm{X}_3$ at 38 82.5
  \pinlabel $\mathrm{X}_1$ at 146.5 97
\pinlabel $\mathrm{O}_2$ at 161.5 82.5
\pinlabel $\mathrm{X}_2$ at 161.5 67
\pinlabel $\mathrm{O}_3$ at 177 67
\pinlabel $\mathrm{X}_3$ at 177 82.5
    \endlabellist
  \includegraphics[width=7.4cm]{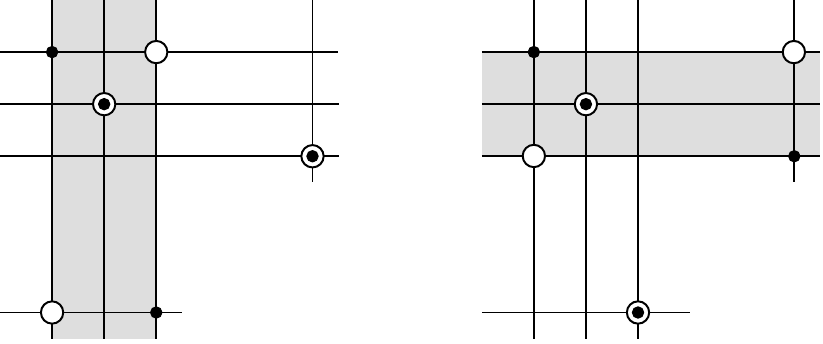}
  \caption{The vertical and horizontal annular domains corresponding to the two cancelling concatenations of the form $r*p$ for $\bx\in\AB$ shown in black and $\by\in\AB$ shown in white.}
  \label{fig:birth1}
\end{figure}
\end{proof}


\begin{lemma}
\label{lem:mapstobirth}
$\phi$ sends $\xpm(\gdt)$ to $\xpm(\gdb)$.
\end{lemma} 

\begin{proof}
Note that $\xpm(\gdt)\in\NB$, so \[\phi(\xpm(\gdt)) = e(\psi(\xpm(\gdt))).\] There is a unique rectangle contributing to each of $\psi(\xp(\gdt))$ and $\psi(\xm(\gdt))$, as shown in Figure \ref{fig:birth2}. It is then clear from the figure that \[e(\psi(\xpm(\gdt)))=\xpm(\gdb).\] \end{proof}

\begin{figure}[htbp]
 \labellist
  \tiny
  \pinlabel $\mathrm{X}_1$ at 16 56
\pinlabel $\mathrm{O}_2$ at 30.5 41.5
\pinlabel $\mathrm{X}_2$ at 30.5 26
\pinlabel $\mathrm{O}_3$ at 46 26
\pinlabel $\mathrm{X}_3$ at 46 41.5
  \pinlabel $\mathrm{X}_1$ at 184.5 56
\pinlabel $\mathrm{O}_2$ at 199.5 41.5
\pinlabel $\mathrm{X}_2$ at 199.5 26
\pinlabel $\mathrm{O}_3$ at 215 26
\pinlabel $\mathrm{X}_3$ at 215 41.5

  \pinlabel $e$ at 82 52
    \pinlabel $e$ at 251 53
  \pinlabel $\mathrm{X}_1$ at 121 56
    \pinlabel $\mathrm{X}_1$ at 289.5 56
    \endlabellist
  \includegraphics[width=9.6cm]{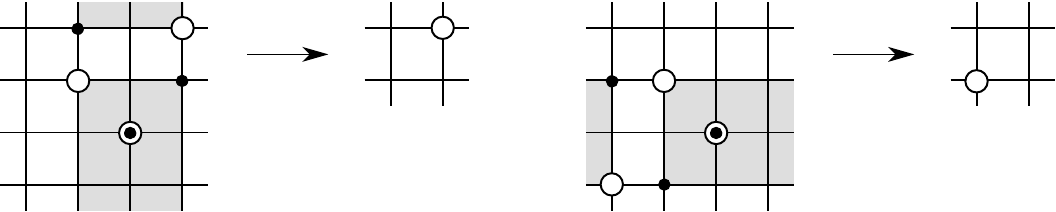}
  \caption{Left, $\xp(\gdt)$  in black and its image under $\psi$  in white. The latter is then sent to $\xp(\gdb)$ by $e$ as indicated by the arrow. Right, the corresponding pictures for $\xm(\gdt)$ and its image under $\phi$. }
  \label{fig:birth2}
\end{figure}
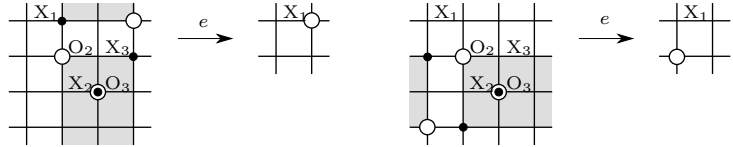

\begin{lemma}
\label{lem:homobirth}
$\phi$ is homogeneous of Maslov--Alexander bidegree $(1,0)$.
\end{lemma} 

\begin{proof}
It is clear from the definition of $\psi$, together with \eqref{eqn:maslov} and \eqref{eqn:alexander}, that $\psi$ is homogeneous. Since $e$ and $\Pi$ are also homogeneous,  the same is true of $\phi$. The bidegree of $\phi$ is then determined by the fact that this map sends $\xp(\gdt)$ to $\xp(\gdb)$,  and the lemma follows immediately from \eqref{eqn:masgrid} and \eqref{eqn:alexgrid}, together with the facts that  \begin{align*}
\tb(\llinkb) &= \tb(\llinkt)+1,\\
\rot(\llinkb)&=\rot(\llinkt),\\
|\llinkb|&=|\llinkt|-1.
\end{align*} 
\end{proof}

The map $\Phit$   induced by $\phi$ therefore satisfies the conclusions of Proposition \ref{prop:birth}. \end{proof}

\subsection{Putting it together}
\label{ssec:proofs}

\begin{proof}[Proof of Theorem \ref{thm:maintilde}]
Suppose $L$ is a decomposable Lagrangian cobordism from $\llinkb$ to $\llinkt$. Then there is a sequence \[\llinkb=\llink_0,\dots,\llink_m = \llinkt\] of Legendrian links such that for each $i$, $\llink_{i+1}$ is obtained from $\llink_{i}$ by either Legendrian isotopy, a pinch, or a birth. Let $\gd_i$ be a grid diagram representing $\llink_i$. Then Proposition \ref{prop:isotopy}, \ref{prop:pinch}, or \ref{prop:birth} provides a map \[\Phit_i:\GHt(\gd_{i+1})\to\GHt(\gd_{i})\] that sends $\lgridtpm(\gd_{i+1})$ to $\lgridtpm(\gd_i)$, for each $i=0,\dots,m-1$. Note that the bidegree of each  $\Phit_i$ may be expressed simply as  \[(\chi(L_i),\frac{1}{2}(\chi(L_i)+|\llink_i|-|\llink_{i+1}|)),\] where $L_i$ is the corresponding elementary cobordism from $\llink_i$ to $\llink_{i+1}$. We  define the map \[\Phit_L:=\Phit_0\circ\dots\circ\Phit_{m-1}:\GHt(\llinkt)\to\GHt(\llinkb).\] This map then sends $\lgridtpm(\gd_+)$ to $\lgridtpm(\gd_-)$ and has bidegree \[(\chi(L),\frac{1}{2}(\chi(L)+|\llinkb|-|\llinkt|)),\] as desired.
\end{proof}

\begin{proof}[Proof of Theorem \ref{thm:main}]
As explained in the introduction, this theorem follows from Theorem \ref{thm:maintilde} and the fact that \[\lgridhpm(\gd) = 0 \textrm{ iff } \lgridtpm(\gd) = 0\] for any grid diagram $\gd$, as in \eqref{eqn:zeroiff}.
\end{proof}

\begin{proof}[Proof of Corollary \ref{cor:lagfilling}]
A decomposable Lagrangian filling of $\llink$ is a decomposable Lagrangian cobordism from the empty link to $\llink$, and can thus  be described as a composition of elementary cobordisms starting with a birth. The rest of the filling is  therefore a decomposable Lagrangian cobordism from the $tb=-1$ Legendrian unknot $\llink_U$ to $\llink$. The corollary then follows from Theorem \ref{thm:main} combined with the fact that $\lgridhpm(\llink_U)\neq 0.$
\end{proof}

\begin{proof}[Proof of Corollary \ref{cor:stab}]
This follows immediately from Theorem \ref{thm:main} and the fact that $\lgridhpm$ vanishes for positive and negative Legendrian stabilizations, respectively, as in Proposition \ref{prop:stab}.
\end{proof}

\section{Examples}
\label{sec:examples}

In this section, we illustrate the effectiveness of Theorem \ref{thm:main} via examples, proving Theorem \ref{thm:examples} along the way.

\subsection{Examples of genus zero}\label{ssec:eg1} Let $K$ denote either one of the  (oriented) smooth knot types given by $m(10_{132})$ or $m(12n_{200})$. In \cite[Section 3]{NgOzsThu08:GRIDEffective}, Ng, Ozsv{\'a}th, and Thurston describe two Legendrian representatives $\llink_0$ and  $\llink_1$ of $K$ with \begin{equation}\label{eqn:tbrex}\tb(\llink_0)=\tb(\llink_1) = -1\textrm{ and } \rot(\llink_0)=\rot(\llink_1)=0,\end{equation} but \begin{equation}\label{eqn:genus01}\lgridhp(\llink_0)=0 \textrm{ and } \lgridhp(\llink_1) \neq 0.\end{equation} It follows that $\llink_0$ and $\llink_1$ are not Legendrian isotopic despite having the same classical invariants, by \cite{OzsSzaThu08:GRID}. These are among     the smallest crossing examples known that demonstrate the effectiveness of the GRID invariants in obstructing Legendrian isotopy.
 Ng, Ozsv{\'a}th, and Thurston further observed, using an  argument by Ng and Traynor from the proof of \cite[Proposition~5.9]{NgTra04:LegTorusLinks}, that these Legendrians are orientation reversals of one another, \[\llink_0=-\llink_1.\] (Note that $m(10_{132})$ and $m(12n_{200})$ are  reversible.) Thus, \cite[Proposition 1.2]{OzsSzaThu08:GRID} implies that \begin{equation}\label{eqn:genus02}\lgridhm(\llink_0) \neq 0 \textrm{ and }\lgridhm(\llink_1)=0\end{equation} as well. Combining \eqref{eqn:genus01} and \eqref{eqn:genus02} with Theorem \ref{thm:main}, we obtain  the following.
 
 \begin{proposition}
 \label{prop:genus0}
There is no decomposable Lagrangian concordance from $\llink_0$ to $\llink_1$ or from $\llink_1$ to $\llink_0$.
 \end{proposition}
 
 \begin{proof}
 The first is obstructed by $\lgridhm$, the second by $\lgridhp$.
 \end{proof} 
Note that the Thurston--Bennequin and rotation numbers do not obstruct the existence of decomposable Lagrangian concordances between $\llink_0$ and $\llink_1$ via \eqref{eqn:classicalobstruction}.

\begin{remark}
One of the two directions in Proposition \ref{prop:genus0} is proven by Baldwin and Sivek \cite{BalSiv18:EqInv} and independently Golla and Juh{\'a}sz in \cite[Proposition 1.7]{GolJuh18:LOSSConc}. In particular, they use the equivalence between $\lgridhp$ and $\lossh$ to obstruct a decomposable Lagrangian concordance from $\llink_1$ to $\llink_0$.
\end{remark}

\subsection{Examples of genus one} Let $K$ denote  one of  $m(10_{132})$ or $m(12n_{200})$, and let $\llink_0$ and $\llink_1$ be the Legendrian representatives of $K$ discussed above. Front diagrams for these Legendrians are given in \cite[Figures 2 and 3]{NgOzsThu08:GRIDEffective}. By modifying these front diagrams for $\llink_0$ and $\llink_1$ first by a Legendrian Reidemeister I move, and then by adding a positive clasp, as  in Figure \ref{fig:modification}, we obtain  new Legendrian knots $\llink_0'$ and $\llink_1'$, respectively, whose front diagrams are shown in Figure \ref{fig:exclasp}. These two Legendrian knots belong to the smooth knot type  \[K' = \begin{cases}
m(12n_{199}),&\textrm{if }K = m(10_{132}),\\
m(14n_{5047}),& \textrm{if }K = m(12n_{200}).
\end{cases}
\]
(The  knot types were found using the program Knotscape by Hoste and Thistlethwaite \cite{HosThi99:Knotscape}.)
\begin{figure}[htbp]

  \includegraphics[width=6.3cm]{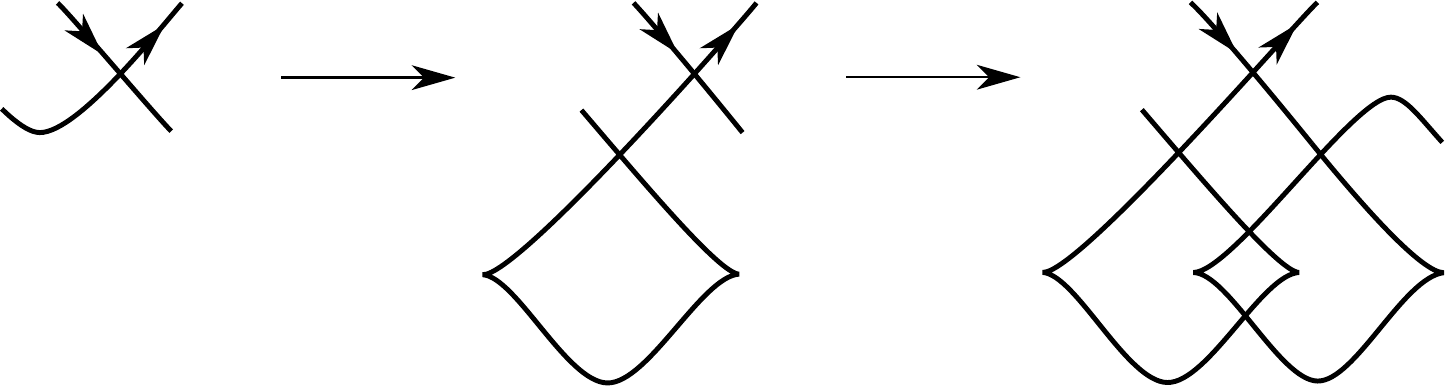}
  \caption{A local modification of the front diagram of $\llinki$.  The first move is a Legendrian Reidemeister I move; the second introduces a positive clasp.}
  \label{fig:modification}
\end{figure}

\begin{figure}[htbp]
   \labellist\tiny
  \pinlabel $X=\{13,6,11,7,1,2,3,4,10,12,8,9,5\}$ at 200 920
  \pinlabel $O=\{8,12,3,4,6,5,1,2,13,9,11,7,10\}$ at 200 880
  
    \pinlabel $X=\{13,8,11,9,3,4,5,6,10,2,7,12,1\}$ at 960 920
  \pinlabel $O=\{10,12,5,6,8,7,3,4,1,9,13,2,11\}$ at 960 880
  
    \pinlabel $X=\{15,8,13,9,6,7,1,2,3,4,12,14,10,11,5\}$ at 200 95
  \pinlabel $O=\{10,14,3,7,8,4,6,5,1,2,15,11,13,9,12\}$ at 200 55
  
    \pinlabel $X=\{15,10,13,11,8,9,3,4,5,6,12,2,7,14,1\}$ at 960 95
  \pinlabel $O=\{12,14,5,9,10,6,8,7,3,4,1,11,15,2,13\}$ at 960 55

    \endlabellist
  \includegraphics[width=10.5cm]{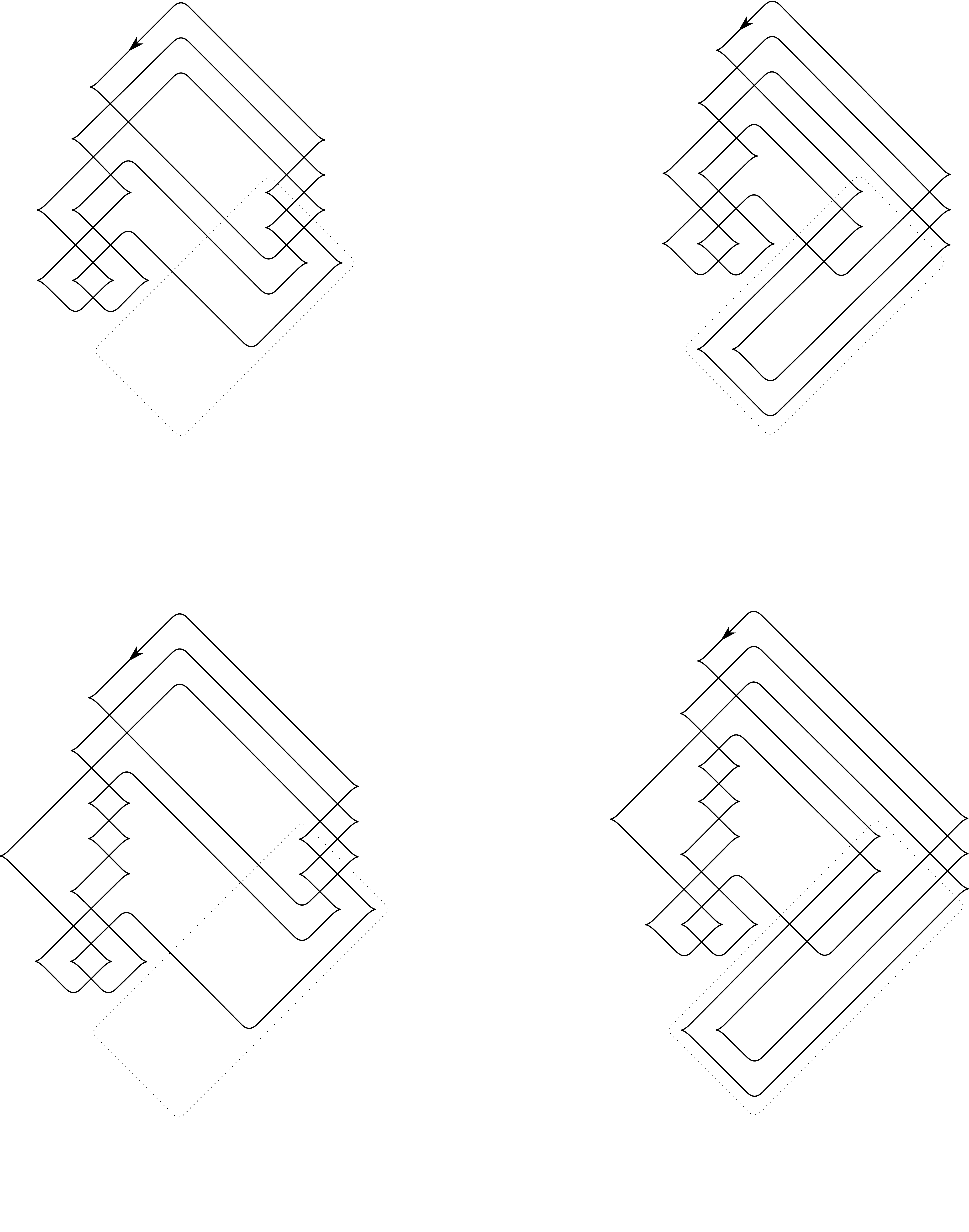}
  \caption{Top, two Legendrian representatives of $m(12n_{199})$. Bottom, two Legendrian representatives of $ m(14n_{5047}).$ In each case, $\llink_0'$ is shown on the left and $\llink_1'$ on the right. Note that  $\llink_0'$ and $\llink_1'$ differ only in the dashed boxes. We have also included the XO coordinates of the corresponding grid diagrams.}
  \label{fig:exclasp}
\end{figure}
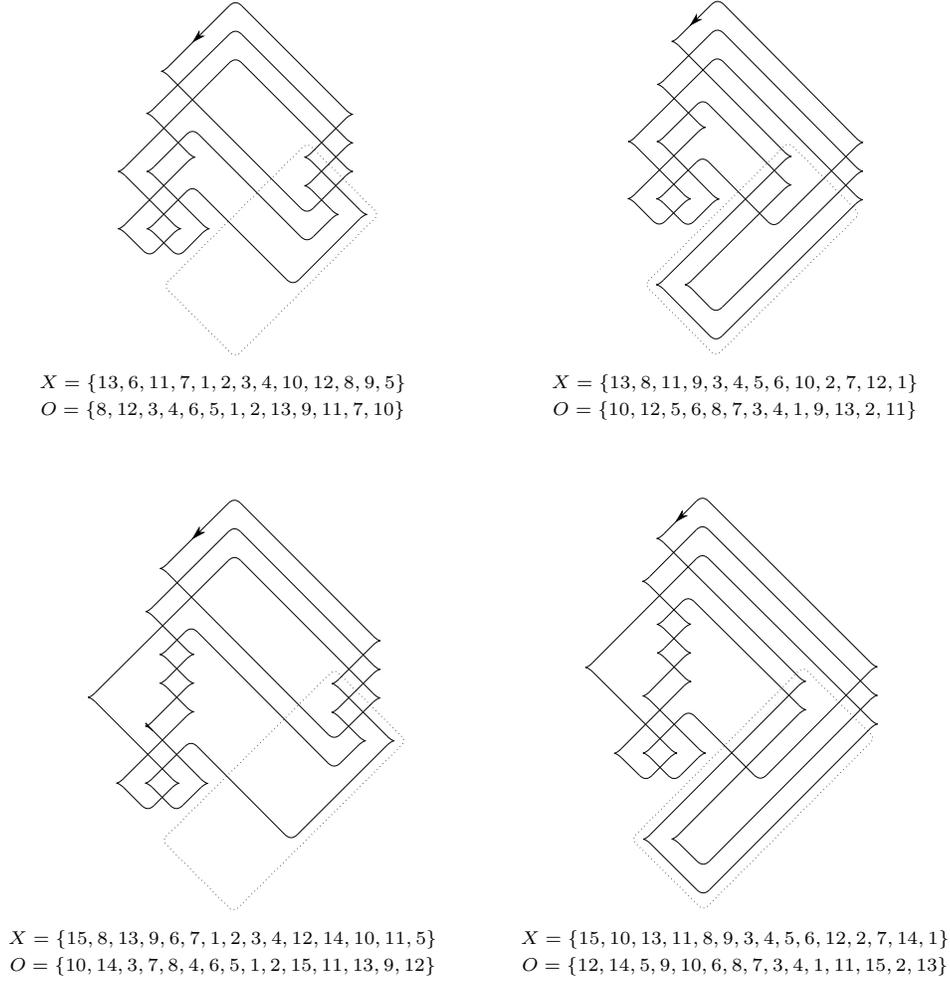

Observe that  there is a smooth cobordism of genus one between $K$ and $K'$ since the latter is obtained from the former via the addition of a positive clasp. Moreover, it is easy to see that the local modification  in Figure \ref{fig:modification} increases the Thurston--Bennequin number by $2$ and preserves the rotation number. Combined with \eqref{eqn:tbrex}, this implies that for any $i,j\in\{0,1\}$, \[\tb(\llink_i')=\tb(\llink_j)+2 \textrm{ and } \rot(\llink_i')=\rot(\llink_j).\] In particular, the classical invariants and smooth topology do not obstruct the existence of a decomposable genus one Lagrangian cobordism from $\llink_0$ to $\llink_1'$ or from $\llink_1$ to $\llink_0'$.

However, a  direct computer calculation using the  program \cite{MeyQuaRob19:TransverseHFK2},\footnote{This program is a newer version of the program written by Ng, Ozsv{\'a}th, and Thurston \cite{NgOzsThu07:TransverseHFK}, with minor bug fixes and improvements in computational efficiency.} applied to the diagrams in Figure \ref{fig:exclasp}, shows that \[\lgridhp(\llink_0')=0 \textrm{ and }\lgridhp(\llink_1')\neq 0.\] The  argument  used  by Ng, Ozsv{\'a}th, and Thurston to show that $\llink_0=-\llink_1$  shows that $\llink_0'$ and $\llink_1'$ are also orientation reversals of one another, which then implies that \[\lgridhm(\llink_0')\neq 0 \textrm{ and }\lgridhm(\llink_1')= 0.\] These calculations, combined with Theorem \ref{thm:main}, lead immediately to the following.

 \begin{proposition}
 \label{prop:genus1}
There is no decomposable Lagrangian cobordism from $\llink_0$ to $\llink_1'$ or from $\llink_1$ to $\llink_0'$.
 \end{proposition}
 
 \begin{proof}
 The first is obstructed by $\lgridhm$, the second by $\lgridhp$.
 \end{proof} 

\subsection{An infinite family}
\label{ssec:infinite}
Let $K$ be  one of  the  knot types $m(10_{145})$, $m(10_{161})$, or $12n_{591}$.  In \cite[Proposition~6]{ChoNg13:LegAtlas}, Chongchitmate and Ng provide  two Legendrian representatives $\llink_0$ and $\llink_1$ of $K$ (denoted by $L_1$ and $L_3$ there) satisfying \begin{equation}\label{eqn:tbinf}\tb(\llink_0)=\tb(\llink_1)+2 \textrm{ and } \rot(\llink_0) = \rot(\llink_1)\end{equation} and \begin{equation}\label{eqn:gridinf}\lgridhp(\llink_1)\neq 0 \textrm{ and } \lgridhm(\llink_1)\neq 0.\end{equation} 
Fix any positive crossing  in the front diagram for $\llink_0$. Let $\llink_0'$ be the Legendrian knot representing the smooth knot type $K'$ obtained from $\llink_0$ by first performing a Legendrian Reidemeister I move near this crossing, and then adding $m$ positive clasps, as  in Figure \ref{fig:clasps}. It is easy to see that this local modification increases the Thurston--Bennequin number by $2m$ and preserves rotation number, \begin{equation}\label{eqn:tbrinf}\tb(\llink_0')=\tb(\llink_0)+2m \textrm{ and } \rot(\llink_0')=\rot(\llink_0).\end{equation} This implies, by \eqref{eqn:tbrstab} combined with \eqref{eqn:tbinf} and \eqref{eqn:tbrinf}, that \begin{equation}\label{eqn:tbrinf2}\tb(S_+(S_-(\llink_0')))= \tb(\llink_1)+2m \textrm{ and }\rot(S_+(S_-(\llink_0')))=\rot(\llink_1).\end{equation}

\begin{figure}[htbp]
  \includegraphics[width=8.5cm]{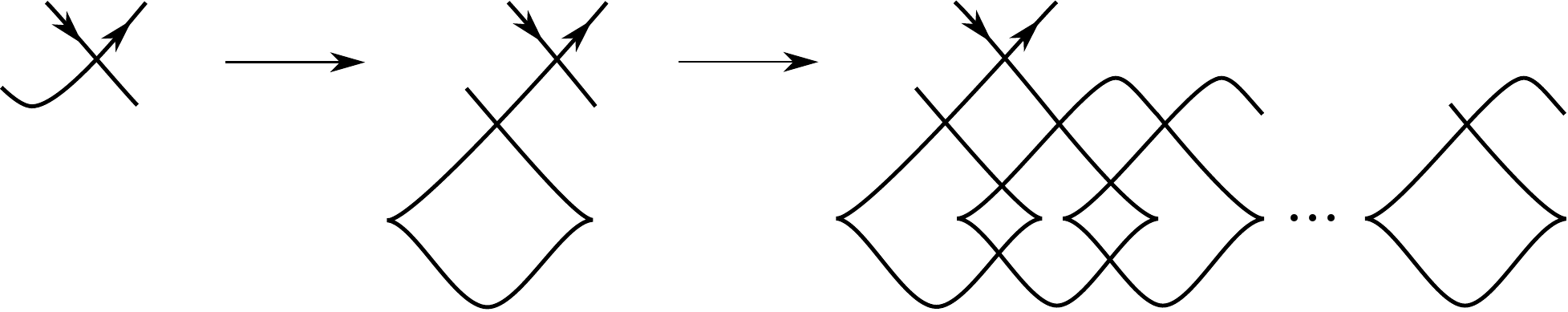}
  \caption{A  modification of the front diagram of $\llink_0$ near a positive crossing.  The first move is a Legendrian Reidemeister I move; the second introduces $m$ positive clasps.}
  \label{fig:clasps}
\end{figure}

With this, we may now prove Theorem \ref{thm:examples}.

\begin{proof}[Proof of Theorem \ref{thm:examples}] Let us adopt the notation from above. Take $g=m$ and let \[\llinkb = \llink_1 \textrm{ and } \llinkt=S_+(S_-(\llink_0')).\] The first is a Legendrian representative of $K$, the second of $K'$, and there is a smooth genus $g$ cobordism from $K$ to $K'$ since $K'$ is obtained from $K$ by adding $g$ clasps, fulfilling the first bullet point of the theorem. The second bullet point is fulfilled by \eqref{eqn:tbrinf2}. Finally, \[\lgridhp(\llinkt)=0\] by Proposition \ref{prop:stab} since $\llinkt$ is a positive Legendrian stabilization, and \[\lgridhp(\llinkb)\neq 0\] by \eqref{eqn:gridinf}, fulfilling the third bullet point of the theorem.
\end{proof}

\subsection{DGA versus GRID}
\label{ssec:dgagrid} We assume below that the reader is familiar with the 
Chekanov--Eliashberg DGA; for a survey, see \cite{EtnyreNg}. 

As mentioned in the introduction, an exact Lagrangian cobordism from $\llinkb$ to $\llinkt$ induces a DGA morphism \cite{EHK} \[(\mathcal{A}_{\llinkt},\partial_{\llinkt})\to(\mathcal{A}_{\llinkb},\partial_{\llinkb}),\] so if the first DGA is trivial while the second is not then there cannot be such a cobordism. The DGA is trivial for stabilized Legendrians \cite{Chekanov}, so, as noted in Section \ref{ssec:examples}, this functoriality cannot obstruct decomposable Lagrangian cobordisms between stabilizations of the examples in Sections \ref{ssec:eg1}-\ref{ssec:infinite}, while the GRID invariants can.



We mentioned in Section \ref{ssec:examples} that there are also cases in which the DGA obstruction applies where the GRID obstruction does not. For example, there is a Legendrian representative $\llinkb$ of the figure eight  with \[\tb(\llinkb)=-3\textrm{ and } \rot(\llinkb)=0,\] such that $(\mathcal{A}_{\llinkb},\partial_{\llinkb})$ admits an \emph{augmentation}, and is therefore nontrivial (see e.g.\ \cite{ChoNg13:LegAtlas}). This example was pointed out to the authors by Steven Sivek. Now, let $\llinkt$ be the Legendrian representative of the right-handed trefoil with \[\tb(\llinkt)=-1\textrm{ and } \rot(\llinkt)=0,\] obtained by stabilizing the $\tb=1$ representative twice, once with each sign, so that $(\mathcal{A}_{\llinkt},\partial_{\llinkt})$ is trivial. These DGAs  therefore obstruct an exact Lagrangian cobordism  from $\llinkb$ to $\llinkt$. There is a smooth genus one cobordism from the figure eight to the right-handed trefoil, as indicated in Figure \ref{fig:dgagrid}, so the classical invariants and smooth topology do not obstruct such a cobordism. On the other hand, the  figure eight has trivial Heegaard Floer tau-invariant, so that \[\tb(\llinkb)+|\rot(\llinkb)|<2\tau(\llinkb)-1.\] Since the figure eight is also \emph{thin},  this implies that the GRID invariants of $\llinkb$ vanish \cite[Proposition 3.4]{NgOzsThu08:GRIDEffective}, and therefore do not obstruct a decomposable Lagrangian cobordism from $\llinkb$ to $\llinkt$ via Theorem \ref{thm:main}.

\begin{figure}[htbp]
 \labellist
   \pinlabel $=$ at 99 30
    \endlabellist
  \includegraphics[width=7cm]{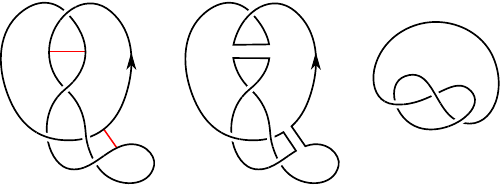}
  \caption{Two oriented band moves certifying the existence of a genus one cobordism between the figure eight knot and the right-handed trefoil.}
  \label{fig:dgagrid}
\end{figure}

Finally, in the interest of completeness, we remark that the DGAs of the  Legendrian representatives of $m(10_{132})$, $m(12n_{200})$, $m(10_{145})$, $m(10_{161})$, and $12n_{591}$, which served as examples of $\llinkb$ in Sections \ref{ssec:eg1}-\ref{ssec:infinite} admit no augmentations and therefore no linearized Legendrian contact homologies. Indeed, Rutherford showed in \cite{Rutherford} that the Kauffman polynomial bound on $\tb (\llink)$ is sharp if and only if $(\mathcal{A}_{\llink}, \partial_{\llink})$ admits an augmentation, and one can check on KnotInfo \cite{KnotInfo} that the Kauffman bounds are not sharp for the Legendrians above. This does  not completely rule out the possibility that their DGAs are nontrivial   (and could therefore perhaps obstruct the Lagrangian cobordisms we are considering), but it  eliminates  the most  tractable approach to proving nontriviality. Pan proved in \cite[Theorem 1.6]{Pan} that if there exists an exact Lagrangian cobordism (with Maslov number 0) from $\llinkb$ to $\llinkt$, then the number of graded augmentations of $\llinkb$ up to a certain equivalence is less than or equal to the number of graded augmentations of $\llinkt$ up to equivalence.  The fact that the Legendrians above admit no augmentations at all also rules out the possibility of applying Pan's more refined obstruction in these examples.

\bibliographystyle{mwamsalphack}
\bibliography{bibliography}

\newcommand{\etalchar}[1]{$^{#1}$}
\providecommand{\bysame}{\leavevmode\hbox to3em{\hrulefill}\thinspace}
\providecommand{\MR}{\relax\ifhmode\unskip\space\fi MR }
\providecommand{\MRhref}[2]{%
  \href{http://www.ams.org/mathscinet-getitem?mr=#1}{#2}
}
\providecommand{\href}[2]{#2}
\begin{thebibliography}{MOSzT07}

\bibitem[BS18a]{BalSiv14:KHMLeg}
John~A. Baldwin and Steven Sivek, \emph{Invariants of {L}egendrian and
  transverse knots in monopole knot homology}, J. Symplectic Geom. \textbf{16}
  (2018), no.~4, 959--1000. \MR{3917725}

\bibitem[BS18b]{BalSiv18:EqInv}
\bysame, \emph{On the equivalence of contact invariants in sutured {F}loer
  homology theories}, preprint, version 2, 2018,
  \href{http://arxiv.org/abs/1601.04973}{\texttt{arXiv:1601.04973}}.

\bibitem[BST15]{BST}
Fr\'{e}d\'{e}ric Bourgeois, Joshua~M. Sabloff, and Lisa Traynor,
  \emph{Lagrangian cobordisms via generating families: construction and
  geography}, Algebr. Geom. Topol. \textbf{15} (2015), no.~4, 2439--2477.
  \MR{3402346}

\bibitem[BVV13]{BalVelVer13:BRAID}
John~A. Baldwin, David~Shea Vela-Vick, and Vera V{\'e}rtesi, \emph{On the
  equivalence of {L}egendrian and transverse invariants in knot {F}loer
  homology}, Geom. Topol. \textbf{17} (2013), no.~2, 925--974. \MR{3070518}

\bibitem[CDGG15]{CDGGG}
Baptiste Chantraine, Georgios Dimitroglou~Rizell, Paolo Ghiggini, and Roman
  Golovko, \emph{Floer homology and {L}agrangian concordance}, Proceedings of
  the {G}\"{o}kova {G}eometry-{T}opology {C}onference 2014, G\"{o}kova
  Geometry/Topology Conference (GGT), G\"{o}kova, 2015, pp.~76--113.
  \MR{3381440}

\bibitem[CET19]{ConEtnTos19:SympFillings}
James Conway, John~B. Etnyre, and B{\"u}lent Tosun, \emph{Symplectic fillings,
  contact surgeries, and {L}agrangian disks}, preprint, version 2, 2019,
  \href{http://arxiv.org/abs/1712.07287}{\texttt{arXiv:1712.07287}}.

\bibitem[Cha10]{Cha10:LagConc}
Baptiste Chantraine, \emph{Lagrangian concordance of {L}egendrian knots},
  Algebr. Geom. Topol. \textbf{10} (2010), no.~1, 63--85. \MR{2580429}

\bibitem[Cha12]{Cha12:DecompLag}
\bysame, \emph{Some non-collarable slices of {L}agrangian surfaces}, Bull.
  Lond. Math. Soc. \textbf{44} (2012), no.~5, 981--987. \MR{2975156}

\bibitem[Cha15]{chantrainesymmetric}
\bysame, \emph{Lagrangian concordance is not a symmetric relation}, Quant.
  Topol. \textbf{6} (2015), no.~3, 451--474.

\bibitem[Che02]{Chekanov}
Yuri Chekanov, \emph{Differential algebra of {L}egendrian links}, Inventiones
  \textbf{150} (2002), no.~3, 441--483.

\bibitem[CL]{KnotInfo}
J.~C. Cha and C.~Livingston, \emph{Knot{I}nfo: {T}able of {K}not {I}nvariants},
  available at \url{https://www.indiana.edu/~knotinfo/}, accessed on Jun 18,
  2019.

\bibitem[CN13]{ChoNg13:LegAtlas}
Wutichai Chongchitmate and Lenhard Ng, \emph{An atlas of {L}egendrian knots},
  Exp. Math. \textbf{22} (2013), no.~1, 26--37. \MR{3038780}

\bibitem[CNS16]{CorNgSiv16:LagConcObstructions}
Christopher Cornwell, Lenhard Ng, and Steven Sivek, \emph{Obstructions to
  {L}agrangian concordance}, Algebr. Geom. Topol. \textbf{16} (2016), no.~2,
  797--824. \MR{3493408}

\bibitem[Dim16]{DR}
Georgios Dimitroglou~Rizell, \emph{Legendrian ambient surgery and {L}egendrian
  contact homology}, J. Symplectic Geom. \textbf{14} (2016), no.~3, 811--901.
  \MR{3548486}

\bibitem[EGH00]{EGH}
Yasha Eliashberg, Alexander Givental, and Helmut Hofer, \emph{Introduction to
  symplectic field theory}, Geom. Funct. Anal. \textbf{Special Volume, Part II}
  (2000), 560--673.

\bibitem[EHK16]{EHK}
Tobias Ekholm, Ko~Honda, and Tam\'{a}s K\'{a}lm\'{a}n, \emph{Legendrian knots
  and exact {L}agrangian cobordisms}, J. Eur. Math. Soc. (JEMS) \textbf{18}
  (2016), no.~11, 2627--2689. \MR{3562353}

\bibitem[Eli98]{Eliashberg}
Yakov Eliashberg, \emph{Invariants in contact topology}, Proceedings of the
  {I}nternational {C}ongress of {M}athematicians, {V}ol. {II} ({B}erlin, 1998),
  no. Extra Vol. II, 1998, pp.~327--338. \MR{1648083 (2000a:57068)}

\bibitem[EN18]{EtnyreNg}
John~B. Etnyre and Lenhard Ng, \emph{Legendrian contact homology in
  $\mathbb{R}^3$}, preprint, version 3, 2018,
  \href{http://arxiv.org/abs/1811.10966}{\texttt{arXiv:1811.10966}}.

\bibitem[GJ19]{GolJuh18:LOSSConc}
Marco Golla and Andr{\'a}s Juh{\'a}sz, \emph{Functoriality of the {$\it EH$}
  class and the {LOSS} invariant under {L}agrangian concordances}, preprint,
  version 2, 2019,
  \href{http://arxiv.org/abs/1801.03716}{\texttt{arXiv:1801.03716}}.

\bibitem[HT99]{HosThi99:Knotscape}
Jim Hoste and Morwen Thistlethwaite, \emph{Knotscape}, version 1.01, 1999,
  available at \url{http://www.math.utk.edu/~morwen/knotscape.html}, accessed
  on Feb 7, 2019.

\bibitem[Juh16]{Juh16:SFHFunctoriality}
Andr{\'a}s Juh{\'a}sz, \emph{Cobordisms of sutured manifolds and the
  functoriality of link {F}loer homology}, Adv. Math. \textbf{299} (2016),
  940--1038. \MR{3519484}

\bibitem[JZ19]{JuhZem18:HFKFunctorialityEquiv}
Andr{\'a}s Juh{\'a}sz and Ian Zemke, \emph{Contact handles, duality, and
  sutured {F}loer homology}, preprint, version 2, 2019,
  \href{http://arxiv.org/abs/1803.04401}{\texttt{arXiv:1803.04401}}.

\bibitem[LOSSz09]{LisOzsSti09:LOSS}
Paolo Lisca, Peter Ozsv{{\'a}}th, Andr{{\'a}}s~I. Stipsicz, and Zolt{{\'a}}n
  Szab{{\'o}}, \emph{Heegaard {F}loer invariants of {L}egendrian knots in
  contact three-manifolds}, J. Eur. Math. Soc. (JEMS) \textbf{11} (2009),
  no.~6, 1307--1363. \MR{2557137 (2010j:57016)}

\bibitem[MOS09]{ManOzsSar09:GH}
Ciprian Manolescu, Peter Ozsv{{\'a}}th, and Sucharit Sarkar, \emph{A
  combinatorial description of knot {F}loer homology}, Ann. of Math. (2)
  \textbf{169} (2009), no.~2, 633--660. \MR{2480614 (2009k:57047)}

\bibitem[MOSzT07]{ManOzsSza07:GHComb}
Ciprian Manolescu, Peter Ozsv{{\'a}}th, Zolt{{\'a}}n Szab{{\'o}}, and Dylan
  Thurston, \emph{On combinatorial link {F}loer homology}, Geom. Topol.
  \textbf{11} (2007), 2339--2412. \MR{2372850 (2009c:57053)}

\bibitem[MQR{\etalchar{+}}19]{MeyQuaRob19:TransverseHFK2}
Lucas Meyers, Robert~John Quarles, Brandon Roberts, David~Shea Vela-Vick, and
  C.-M.~Michael Wong, \emph{transverse-hfk-revision}, version 1.0.0, 2019,
  available at \url{https://github.com/albenzo/transverse-hfk-revision/},
  accessed on Jun 28, 2019.

\bibitem[NOT07]{NgOzsThu07:TransverseHFK}
Lenhard Ng, Peter Ozsv{{\'a}}th, and Dylan Thurston,
  \emph{Transverse{H}{F}{K}.c}, 2007, available at
  \url{https://services.math.duke.edu/~ng/math/TransverseHFK.c}, accessed on
  Feb 8, 2019.

\bibitem[NOT08]{NgOzsThu08:GRIDEffective}
\bysame, \emph{Transverse knots distinguished by knot {F}loer homology}, J.
  Symplectic Geom. \textbf{6} (2008), no.~4, 461--490. \MR{2471100
  (2009j:57014)}

\bibitem[NT04]{NgTra04:LegTorusLinks}
Lenhard Ng and Lisa Traynor, \emph{Legendrian solid-torus links}, J. Symplectic
  Geom. \textbf{2} (2004), no.~3, 411--443. \MR{2131643}

\bibitem[OSSz15]{OzsStiSza15:GHBook}
Peter~S. Ozsv{\'a}th, Andr{\'a}s~I. Stipsicz, and Zolt{\'a}n Szab{\'o},
  \emph{Grid homology for knots and links}, Mathematical Surveys and
  Monographs, vol. 208, American Mathematical Society, Providence, RI, 2015.
  \MR{3381987}

\bibitem[OSzT08]{OzsSzaThu08:GRID}
Peter Ozsv{{\'a}}th, Zolt{{\'a}}n Szab{{\'o}}, and Dylan Thurston,
  \emph{Legendrian knots, transverse knots and combinatorial {F}loer homology},
  Geom. Topol. \textbf{12} (2008), no.~2, 941--980. \MR{2403802 (2009f:57051)}

\bibitem[Pan17]{Pan}
Yu~Pan, \emph{The augmentation category map induced by exact {L}agrangian
  cobordisms}, Algebr. Geom. Topol. \textbf{17} (2017), no.~3, 1813--1870.
  \MR{3677941}

\bibitem[Rut06]{Rutherford}
Dan Rutherford, \emph{The {T}hurston-{B}ennequin number, {K}auffman polynomial,
  and ruling invariants of a {L}egendrian link: {T}he {F}uchs conjecture and
  beyond}, Int. Math. Res. Not. (2006), Art. ID 78591, 15. \MR{2219227}

\bibitem[ST13]{ST}
Joshua~M. Sabloff and Lisa Traynor, \emph{Obstructions to {L}agrangian
  cobordisms between {L}egendrians via generating families}, Algebr. Geom.
  Topol. \textbf{13} (2013), no.~5, 2733--2797. \MR{3116302}

\bibitem[Won17]{Won17:GHUnorientedSkein}
C.-M.~Michael Wong, \emph{Grid diagrams and {M}anolescu's unoriented skein
  exact triangle for knot {F}loer homology}, Algebr. Geom. Topol. \textbf{17}
  (2017), no.~3, 1283--1321. \MR{3677929}

\bibitem[Zem19]{Zem19:HFKFunctoriality}
Ian Zemke, \emph{Link cobordisms and functoriality in link {F}loer homology},
  J. Topol. \textbf{12} (2019), no.~1, 94--220.

\end{thebibliography}

\end{document}